\documentclass[reqno,english]{amsart}

\usepackage{amssymb,latexsym,amsmath,amsthm}
\usepackage{graphicx,color,palatino,mathpazo, amsfonts}

\usepackage[colorlinks=true]{hyperref}
\hypersetup{citecolor=blue}

\numberwithin{equation}{section}
\newtheorem{theorem}{Theorem}[section]
\newtheorem{proposition}[theorem]{Proposition}
\newtheorem{lemma}[theorem]{Lemma}

\theoremstyle{definition}

\def\XXint#1#2#3{{\setbox0=\hbox{$#1{#2#3}{\int}$}
     \vcenter{\hbox{$#2#3$}}\kern-.5\wd0}}

\def\p{\partial}

\def\R{\mathbb R}

\newcommand{\ve }{\varepsilon }

\newcommand{\be}{\begin{equation}}
\newcommand{\ee}{\end{equation}}
\newcommand{\equ}[1]{(\ref{#1})}

\long\def\hide#1{}

\begin{document}

\title[Travelling and rotating solutions for SQG]{Travelling and rotating solutions to the generalized inviscid surface quasi-geostrophic equation}

\author[W. Ao]{Weiwei Ao}
\address{Weiwei Ao
\hfill\break\indent
Wuhan University
\hfill\break\indent
Department of Mathematics and Statistics, Wuhan, 430072, PR China}
\email{wwao@whu.edu.cn}

\author[J. D\'avila]{Juan D\'avila}
\address{
Juan D\'avila \hfill \break \indent
Department of Mathematical Sciences University of Bath,
Bath BA2 7AY, United Kingdom, and
Instituto de Matem\'aticas, Universidad de Antioquia, Calle 67, No. 53--108, Medell\'\i{}n, Colombia}

\email{jddb22@math.ac.uk}

\author[M. del Pino]{Manuel del Pino}

\address{Manuel del Pino \hfill \break \indent
Department of Mathematical Sciences, University of Bath, Bath, Ba2 7AY, UK}

\email{m.delpino@bath.ac.uk}

\author[M.Musso]{Monica Musso}

\address{ Monica Musso \hfill \break \indent Department of Mathematical Sciences, University of Bath, Bath, Ba2 7AY, UK. }

\email{m.musso@bath.ac.uk}

\author[J. Wei]{Juncheng Wei}
\address{Juncheng Wei
\hfill\break\indent
University of British Columbia
\hfill\break\indent
 Department of Mathematics, Vancouver, BC V6T1Z2, Canada}
 \email{jcwei@math.ubc.ca}

\date{\today}

\begin{abstract}
For the generalized surface quasi-geostrophic equation
$$
\left\{
\begin{aligned}
& \partial_t \theta+u\cdot \nabla \theta=0,
\quad \text{in } \R^2 \times (0,T),
\\
& u=\nabla^\perp \psi,
\quad
\psi = (-\Delta)^{-s}\theta
\quad \text{in } \R^2 \times (0,T) ,
\end{aligned}
\right.
$$
$0<s<1$, we consider for $k\ge1$ the problem of finding a family of
$k$-vortex solutions $\theta_\ve(x,t)$ such that as $\ve\to 0$
$$
\theta_\ve(x,t) \rightharpoonup \sum_{j=1}^k m_j\delta(x-\xi_j(t))
$$
for suitable trajectories for the vortices $x=\xi_j(t)$. We find such solutions in the special cases of vortices travelling with constant speed along one axis or rotating with same speed around the origin. In those cases the problem is reduced to a fractional elliptic equation which is treated with singular perturbation methods.  A key element in our construction is a proof of the non-degeneracy of the radial ground state for the so-called fractional plasma problem
$$ (-\Delta)^sW = (W-1)^\gamma_+ , \quad \text{in } \R^2, \quad 1<\gamma < \frac{1+s}{1-s}$$
whose existence and uniqueness have recently been proven in \cite{chan_uniqueness_2020}.

\end{abstract}

\maketitle

\section{Introduction} 

In this paper we consider the problem
\begin{align}
\label{generalizedsqg}
\left\{
\begin{aligned}
& \partial_t \theta+u\cdot \nabla \theta=0,
\quad \text{in } \R^2 \times (0,T)
\\
& u=\nabla^\perp \psi,
\quad
\psi = (-\Delta)^{-s}\theta
\quad \text{in } \R^2 \times (0,T) ,
\end{aligned}
\right.
\end{align}
where $0<s<1$ and $(a_1,a_2)^\perp = (a_2,-a_1)$,
which is known as the modified or generalized surface quasi-geostrophic equation.
Here $\theta$ is the active scalar being transported by the velocity field $u$ generated by $\theta$ and
$\psi$ is the stream function.
The operator $(-\Delta)^{-s}$ in $\R^n$ is the standard inverse of the fractional laplacian and is given by the expression
\[
(-\Delta)^{-s}\theta(x) =
\int_{\R^n} G_s(x-y) \theta(y)\,ds ,
\quad G_s(z) = \frac{c_{n,s}}{|z|^{n-2s}}
\]
where
\begin{align}
\label{cns}
c_{n,s} = \pi^{-\frac{n}{2}} 2^{-2s} \frac{\Gamma(\frac{n-2s}{2})}{\Gamma(s)}.
\end{align}
The case $s=\frac{1}{2}$ in \eqref{generalizedsqg} corresponds to the surface quasi-geostrophic  (SQG) equation while the limit  $s\uparrow 1$ corresponds to the 2D Euler equation.
In the formulation \eqref{generalizedsqg} we assume that $\theta$ is sufficiently regular so that $\psi$ is $C^1$.




\medskip
Equations \eqref{generalizedsqg} for $s=\frac 12 $ first appeared as models of geophysical flows. After the classical work by Constantin, Majda and Tabak \cite{Constantin_1994}, who pointed out its formal mathematical analogies with the three dimensional Euler equation, these equations have been widely investigated.

\medskip
The Cauchy problem for \eqref{generalizedsqg} is a delicate matter. For $0<s<1$, local well-posedness is known for sufficiently regular initial data,
see \cite{castro_global_2017,chae-et-al-2012, Constantin_1994, kiselev_local_2017} and the references therein. Large class of initial data (patches)  may produce finite time singularities, see \cite{kiselev_finite_2016, kiselev_local_2017}.



\medskip
Of special interest are solutions of \eqref{generalizedsqg} with highly concentrated values of the active scalar $\theta(x,t)$ around a finite number of points $\xi_1(t), \ldots, \xi_k(t)$ which are idealized as regular solutions that approximate a singular object of the form
\be\label{theta}
\sum_{i=1}^k m_i \delta (x-\xi_i(t)) ,
\ee
where $\delta(x)$ is the standard Dirac mass at the origin. The constants $m_i$ are called the intensities of the vortices $\xi_i(t)$.
In the case $s=1$, corresponding to the 2D Euler equation, these solutions represent fluids with sharply concentrated vorticities around the points  $\xi_i(t)$.  In this setting the problem is classical and traces back Kirchhoff.
The location of the limiting point vortices is found by formal substitution, leading to the Hamiltonian system
\be
\dot \xi_j(t) =   \frac 1{2\pi} \sum_{i\ne j} ^k m_i\frac {(\xi_i(t)- \xi_j(t))^\perp }{|\xi_i(t)- \xi_j(t)|^2 } \quad j=1,\ldots, k .
\label{k1}\ee
Finding regular solutions that approximate
the superposition of  point vortices \equ{theta} for a given solution of system \equ{k1},
is the classical  {\em vortex desingularization problem}. See the works \cite{marchioro_euler_1983, davila_gluing_2020-1, SS2013, CLW2014, smets_desingularization_2010} and references therein.

\medskip
For the generalized SQG equation \eqref{generalizedsqg}, the point vortex model corresponding to a solution of the form \equ{theta} becomes
\be\label{aa}
\dot \xi_j(t) =
\frac{1}{2^{2s-1} \pi }
\frac{\Gamma(2-s)}{\Gamma(s)}
\sum_{i\not=j} m_i
\frac{(\xi_i(t)-\xi_j(t))^\perp}{|\xi_i(t)-\xi_j(t)|^{4-2s}} ,\quad j=1,\ldots,k ,
\ee
see the recent work by Rosenzweig \cite{rosenzweig} and references therein.

\medskip
The purpose of this paper is to construct regular solutions $\theta(x,t)$ which resemble a superposition of point vortices of the form  \equ{theta}, where the $k$-tuple
 $(\xi_1(t),\ldots,\xi_k(t))$ represents a solution of system \eqref{aa} which {\em does not change form}  as time evolves.
 More precisely, we focus on traveling and rotating solutions of system \eqref{aa}.

 \medskip
A traveling solution of \eqref{aa} is one of the form
\begin{align}
\label{xi-trav}
\xi_j(t) = b_j + c t e_2
\end{align}
where   $b_1,\ldots,b_k$ are points in $\R^2$, the constant
$c\in \R$ is the speed and without loss of generality we take the travel direction to be $e_2 =(0,1)$.
Then \eqref{aa} reduces to the system
\begin{align}
\label{reduced-trav}
c e_2 = \frac{\Gamma(2-s)}{ 2^{2s-1} \pi \Gamma(s)}
\sum_{i\not=j} m_i
\frac{(b_i-b_j)^\perp}{|b_i-b_j|^{4-2s}} ,\quad j=1,\ldots,k .
\end{align}
A rotating solution of \eqref{aa} is one of the form
\begin{align}
\label{xi-rot}
\xi_j(t) = Q_{\alpha t} b_j , \quad
Q_{\alpha t}
= \left[
\begin{matrix}
\cos(\alpha t) & -\sin(\alpha t) \\
\sin(\alpha t) & \phantom{-}\cos(\alpha t)
\end{matrix}
\right]
\end{align}
and  $b_1,\ldots,b_k \in \R^2$.
These are solutions of \eqref{aa} if
\begin{align}
\label{reduced-rot}
\alpha b_j =
- \frac{\Gamma(2-s)}{ 2^{2s-1} \pi \Gamma(s)}
\sum_{i\not=j} m_i
\frac{b_i-b_j}{|b_i-b_j|^{4-2s}} ,\quad j=1,\ldots,k .
\end{align}


For simplicity, we will concentrate on the most elementary solutions to  \eqref{reduced-trav} and \eqref{reduced-rot}.
For \eqref{reduced-trav} we consider the traveling vortex pair, namely the solution with
$k=2$, and
\begin{align}
\label{trav-param}
\begin{aligned}
b_1 &= d e_1, \quad b_2 = - d e_1, \quad e_1=(1,0), \quad m_1=-m_2=m,
\\
c & =  - \frac{\Gamma(2-s)}{4\pi  \Gamma(s)} \frac{m}{d^{3-2s}} ,
\end{aligned}
\end{align}
where $d>0$.

In the case of rotating solutions, we consider the rotating polygon with equal masses, that is, for $k\geq 2$,
\begin{align}
\label{rot-param}
\begin{aligned}
b_j &= \rho e^{2\pi i \frac{j}{k}} , \quad
m_j = m, \quad j=0,\ldots,k-1,
\\
\alpha &= \frac{m}{\rho^{2-2s}}
\frac{\Gamma(2-s)}{ 2^{s+1}\pi \Gamma(s)}
\sum_{l=1}^{k-1} \frac{1}{(1-\cos(\frac{2\pi l}{k}))^{1-s}},
\end{aligned}
\end{align}
where $\rho>0$.

\subsection{Main results}

In analogy with the solution \eqref{xi-trav}  of \eqref{aa}, we look for traveling solutions to \eqref{generalizedsqg}
by requiring that
\begin{equation}
\label{travel-theta}
\theta(x_1,x_2, t)=\Theta(x_1,x_2-ct) , 
\end{equation}
for some profile function $\Theta(x_1,x_2)$ defined on $\R^2$.
In this case, the generalized SQG equation \eqref{generalizedsqg} can be rewritten as the stationary problem
\begin{equation}\label{orthogonal}
(\nabla^\perp \Psi-ce_2) \cdot  \nabla \Theta =0	,
\quad \Psi = (-\Delta)^{-s} \Theta .
\end{equation}
The condition that $\theta$ approximates \eqref{theta} now becomes
\begin{align}
\label{Theta-conc}
\Theta(x)\approx \sum_{j=1}^k m_j \delta(x-b_j).
\end{align}

\medskip
Similarly, associated to solutions \eqref{xi-rot} of system \equ{aa}, we look for rotating solutions $\theta(x)$ of \eqref{generalizedsqg} close to \eqref{theta}, by requiring that
\begin{align}
\label{rot-theta}
\theta(x,t) = \Theta( Q_{-\alpha t} x) , \quad x\in \R^2 ,
\end{align}
with $\Theta(x)$ also having the concentration behavior \eqref{Theta-conc}.
Then  \eqref{generalizedsqg} becomes
\begin{align}
\label{elliptic-rot}
(\nabla^\perp \Psi + \alpha x^\perp )\cdot \nabla \Theta = 0 , \quad \Psi = (-\Delta)^{-s} \Theta.
\end{align}


Our first result states the existence of a traveling solution concentrated near the vortex pair associated to the solution \eqref{xi-trav}, \eqref{trav-param} of system \eqref{reduced-trav}.

\begin{theorem}
\label{thm-vortex-pair}
Consider the traveling vortex pair given by \eqref{trav-param}. Then for $\varepsilon>0$ small there is a solution $\theta_\varepsilon$ of \eqref{generalizedsqg} of the form \eqref{travel-theta} such that $\Theta_\varepsilon$ is $C^1(\R^2)$, and
\begin{align*}
& \Theta_\varepsilon (x) \rightharpoonup
m \delta(x-b_1)
-m \delta(x-b_2)
\quad \text{as } \varepsilon\to 0,
\quad
\mathop{\mathrm{supp}} \Theta_\varepsilon \subset \bigcup_{j=1}^2 B_{C \varepsilon}(b_j),
\end{align*}
where the convergence is in the sense of measures and $C>0$ is a constant.
\end{theorem}

Similarly, we obtain rotating concentrated solutions near the vertices of the rotating polygon solution \eqref{xi-rot}, \eqref{rot-param} of \eqref{reduced-rot}.
\begin{theorem}
\label{thm-rot}
Consider the rotating polygon given by  \eqref{rot-param}. Then for $\varepsilon>0$ small there is a solution $\theta_\varepsilon$ of \eqref{generalizedsqg} of the form \eqref{rot-theta} such that
 $\Theta_\varepsilon$ is $C^1(\R^2)$,
\begin{align*}
& \Theta_\varepsilon (x) \rightharpoonup m \sum_{j=1}^2  \delta(x-b_j) \quad \text{as } \varepsilon\to 0,
\quad
\mathop{\mathrm{supp}} \Theta_\varepsilon \subset \bigcup_{j=1}^k B_{C \varepsilon}(b_j),
\end{align*}
where the convergence is in the sense of measures and $C>0$ is a constant.
\end{theorem}

A natural way of obtaining solutions  to the stationary problem \eqref{orthogonal} is to locally  impose that $\Theta(x) = f( \Psi(x) + c x_1) $  for a sufficiently regular function $f(u)$ so that \eqref{orthogonal} locally becomes  the elliptic equation
\begin{align}
\label{elliptic1}
(-\Delta)^s \Psi =  f( \Psi + c x_1) .
\end{align}
Similarly, locally  imposing
 $
\Theta(x) = f( \Psi(x)  + \alpha \frac{|x|^2}{2})$,
problem
\eqref{elliptic-rot} becomes
\begin{align}\label{elliptic2}
(-\Delta)^s \Psi =  f\Big( \Psi +\alpha \frac{|x|^2}{2} \Big) .
\end{align}  Using this observation,
 Gravejat and Smets  \cite{gravejat_smooth_2019}  have recently found solutions $\Theta(x)$ to problem
\eqref{orthogonal}  for $s=\frac 12$, with compact support and odd symmetry in $x_1$. They use  variational techniques applied to a suitable class of subcritical nonlinearities.

\subsection{Extensions}
Traveling solutions with multiple vortices can be found under  suitable non-degeneracy conditions for solutions of
system \eqref{reduced-trav}.
For instance, following \cite{liu_multi-vortex_2018}, we can consider configurations with $k$ vortices with intensities 1 located at  points $p_1, \ldots, p_k$
and $k$  vortices with intensities $-1$ located at $q_1, \ldots, q_k$.
We say that $ {\bf b} =  ({\bf p}, {\bf q})=(p_1, \ldots, p_k, q_1,\ldots, q_k)$ is an array of traveling vortices  if it satisfies the following conditions.
\begin{align}
\label{bal}
\left\{
\begin{aligned}
\sum_{j \not = i} \frac{p_i-p_j}{ |p_i-p_j|^{4-2s}}
-\sum_{ l=1 }^k \frac{p_i-q_l}{ |p_i-q_l|^{4-2s}} & = c \frac{2^{2s-1}\pi \Gamma(s)}{\Gamma(2-s)} e_1, \quad  i=1,\ldots, k,\\
\sum_{ l\not =m  } \frac{q_m-q_l}{ |q_m-q_l|^{4-2s}} -\sum_{j=1}^k \frac{q_m-p_j}{ |q_m-p_j|^{4-2s}} & =- c \frac{2^{2s-1}\pi \Gamma(s)}{\Gamma(2-s)} e_1, \quad m=1,\ldots, k,
\end{aligned}
\right.
\end{align}
where  $c\in \R$.
The set of points $ ({\bf p}, {\bf q})$ is called  a symmetric array of traveling vortices if in addition to (\ref{bal}) it satisfies
 \begin{equation}
 \label{sym1}
\left\{\begin{array}{l}
 q_i=-\bar{p}_i, \quad i=1,\ldots, k;\\
  \mbox{there exists} \ j_0 \ \mbox{such that} \ p_{2j-1}=\bar{p}_{2j},\quad  j=1, \ldots, j_0  \\
  \mbox{and} \ \mbox{Im} (p_j)=0,
  \mbox{for} \ j=2j_0+1,\ldots, k .
  \end{array}
  \right.
 \end{equation}
(Here $ \bar{z}$ denotes the conjugate of $z$ and $ Im(z)$ is the imaginary part of $z$.)
A symmetric array of traveling vortices $ ({\bf p}, {\bf q})$  is called nondegenerate if the linearization map at $ ({\bf p}, {\bf q})$ among points satisfying (\ref{sym1}) has only trivial kernel.

When $s=1$ these definitions are introduced in \cite{liu_multi-vortex_2018}.
There it is shown that the roots of certain Adler-Moser polynomials are nondegenerate symmetric arrays of traveling vortices with $k=\frac{n(n+1)}{2}$ for some integer $n$. As a consequence they constructed multiple vortices to the traveling wave equation to Gross-Pitaevskii equation.

By a perturbation argument in $s$, we also obtain that for each fixed $ k=\frac{n(n+1)}{2}$ there exist nondegenerate symmetric arrays of traveling vortices when $ s<1, |1-s| \ll 1$. Since the forces are analytic in $s$, we infer that except finite number of $s$, for each fixed $k=\frac{n(n+1)}{2}$, there exist a unique  nondegenerate symmetric array of traveling vortices. For general $s\in (0,1)$, it is an interesting and challenging question to find nondegenerate symmetric arrays of traveling vortices. In the  special case $s=\frac{1}{2}$ (the SQG case), we can use MatLab\footnote{We thank Prof. Yong Liu for the computations.} to compute numerically the existence of six nondegenerate symmetric arrays of traveling vortices (See Figure 1): $p_1=-\bar{q}_1= (-1.026, 0.563), p_2=-\bar{q}_2= (-1.026, -0.563), p_3=-\bar{q}_3= (0.368, 0). $
\begin{figure}
\caption{Figure 1. Six vortices}
\centering
{\includegraphics[
height=1.6in,
width=2.2in
]{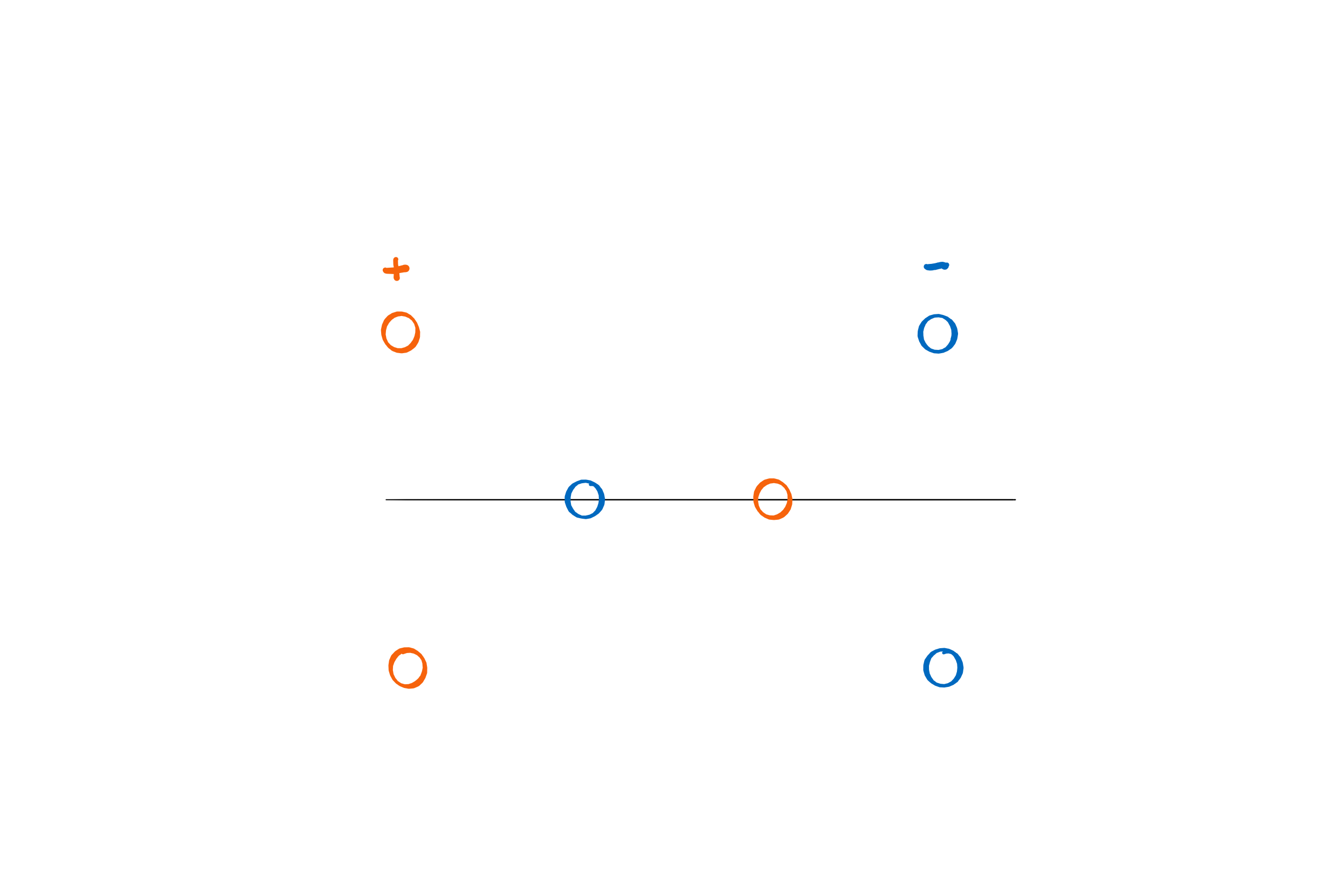}}
\end{figure}

We state the following theorem on the existence of multiple vortex-anti vortex solutions to \eqref{orthogonal}.
\begin{theorem}\label{thm-multiple-trav}
Let $ (p_1, \ldots, p_k, q_1,\ldots, q_k)$ be a nondegenerate symmetric array of traveling vortices. Then for $\ve>0 $ small, there exists a solution
$\Theta_\varepsilon$ to \eqref{orthogonal}
such that
\begin{equation*}
\Theta_\varepsilon(x) \rightharpoonup
\sum_{j=1}^k \delta (x-p_j)
-
\sum_{j=1}^k \delta (x-q_j)
\end{equation*}	
as $\ve \to 0$.	Moreover $\mathop{\mathrm{supp}} \Theta_\varepsilon \subset \bigcup_{j=1}^k B_{C \varepsilon}(b_j)$ and has $\Theta_\varepsilon$ has  the symmetries
\begin{align*}
\Theta_\varepsilon(x_1,x_2)=-\Theta_\varepsilon(-x_1,x_2)=\Theta_\varepsilon(x_1,-x_2)
\quad \text{for all }(x_1,x_2)\in \R^2 .
\end{align*}
\end{theorem}

A similar result can be found for the Euler equation case $s=1$, applying the results of \cite{CLW2014}.

\medskip{
}


More general traveling solutions than those in Theorem~\ref{thm-multiple-trav} can also be found.
Let us consider the functional of $k$ points
$\mathbf{b} = (b_1,\ldots,b_k) \in \R^{2k}$,
\begin{align*}
I({\bf b}) =
c \sum_{i=1}^k  m_i  (b_i\cdot e_1)
+ \frac{\Gamma(1-s)}{\pi 2^{2s-1}\Gamma(s)}
\sum_{i\not=j}
\frac{ m_im_j}{|b_i-b_j|^{2-2s}} .
\end{align*}
Critical points of $I(\bf b)$ correspond to solutions of system \equ{reduced-trav}. The functional $I$ is invariant under translations along the $e_2$-direction, and therefore all critical points of $I$ are degenerate.
We say that a critical point $\bf b$ of $I$ is non-degenerate up to vertical translations, if $D^2 I (\bf b)$ has a one-dimensional kernel.
The following result holds.

\begin{theorem}
\label{thm-trav-general}
If $\bf b$ is a critical point {\em non-degenerate up to vertical translations}  then there exists a solution
$\Theta_\ve(x)$ of Equation \equ{orthogonal}
such that
$\mathop{\mathrm{supp}} \Theta_\varepsilon \subset \bigcup_{j=1}^k B_{C \varepsilon}(b_j)$ and
\[
\Theta_\varepsilon (x)\rightharpoonup \sum_{j=1}^k m_j \delta(x-b_j)
\quad\hbox{as $\varepsilon\to 0$.}
\]
\end{theorem}


\medskip

A similar result holds in the case of rotating solutions.
Let $k\geq 2$ be an integer, $m_1,\ldots,m_k \in \R$ and $\alpha \not=0$.
Let us consider the energy functional
\begin{align*}
 J({\bf b}) = \frac \alpha 2 \sum_{i=1}^k  m_i  |b_i|^2
  + \frac{\Gamma(1-s)}{\pi 2^{2s-1}\Gamma(s)}
\sum_{i\not=j}
\frac{ m_im_j}{|b_i-b_j|^{2-2s}} ,
\end{align*}
where $\mathbf{b} = (b_1,\ldots,b_k) \in \R^{2k}$.
Critical points of $J$ correspond to solutions of \eqref{reduced-rot}. Since this functional is invariant under rotations around the origin, its critical points are always degenerate. We say that a critical point  $\mathbf{b}$  is non-degenerate up to rotations if $D^2 J(\mathbf{b})$ has a one-dimensional kernel.

\begin{theorem}
\label{thm-rot2}
Let $\mathbf{b} = (b_1,\ldots,b_k)$ be a  critical point of $J$ that is non-degenerate up to rotations.
Then for $\varepsilon>0$ small there exists a solution $\Theta_\varepsilon$ to \eqref{elliptic-rot} such that
\[
\Theta_\varepsilon (x)\rightharpoonup \sum_{j=1}^k m_j \delta(x-b_j)
\]
as $\varepsilon\to 0$.
Moreover $\mathop{\mathrm{supp}} \Theta_\varepsilon \subset \bigcup_{j=1}^k B_{C \varepsilon}(b_j)$.
\end{theorem}

In  \cite{campos-delpino-dolbeault}
a related problem in gravitation theory, consisting in the desingularization of rotating point masses in a continuous model of stellar dynamics, has been considered.
The result obtained there is similar to Theorem~\ref{thm-rot2} for $s=\frac{1}{2}$.
Critical points of $J$ for $s=\frac{1}{2} $
are called relative equilibria for the $N$-body problem. Their study is  classical  in celestial mechanics. In particular, it is known that for almost every choice of masses, critical points are non-degenerate up rotations and their number is estimated, see \cite{palmore-annals}. See also \cite{moulton, perko-walter,palmore-class1,palmore-class2,palmore-class3} and references therein.

\subsection{The fractional plasma problem}

The proof of  theorems~\ref{thm-vortex-pair}--\ref{thm-rot2} consists of gluing highly concentrated solutions of special elliptic equations of the form \eqref{elliptic1} and \eqref{elliptic2}.
For this purpose we will use suitably scaled radial solutions of the so-called fractional plasma problem.
In any space dimension $n\geq 2$, this is the semi-linear elliptic problem
\begin{align}
\label{plasma0}
\left\{
\begin{aligned}
(-\Delta)^s W &=(W - 1)_+^\gamma \quad \mbox{ in }\R^n	,
\\
W(x)& \to 0 \quad \text{as }|x|\to \infty ,
\end{aligned}
\right.
\end{align}
where $s\in (0,1)$, $1\leq \gamma < \frac{n+2s}{n-2s}$.
Problem (\ref{plasma0}) arises in the context of aggregation-diffusion equations, see \cite{carrillo_ground_2018}.



When  $s=1$ this free boundary from problem has been studied in  \cite{temam_non-linear_1975,temam_remarks_1977,bandle_harmonic_1996,caffarelli_asymptotic_1980,flucher_asymptotic_1998}.


\medskip
For $s=1$ solutions are radially symmetric up to translations and can be analyzed using ODE techniques. This is not possible in the nonlocal case  $s\in (0,1)$ and the analysis
becomes substantially harder.
Recently, in \cite{chan_uniqueness_2020} it has been proven that
(\ref{plasma0}) has a unique radial solution.
The proof relies on an application of a monotonicity formula developed for the fractional Schr\"odinger equation in  \cite{frank_uniqueness_2016}.

In this paper, we will use this ground state to (\ref{plasma0}) to construct solutions  by a Lyapunov-Schmidt reduction.
Our first contribution in this paper is to derive the non-degeneracy of the radial ground state solution of \eqref{plasma0}.






\medskip

The paper is organized as follows.
In section \ref{sect-scheme} we describe the elliptic equation we use to prove Theorems~\ref{thm-vortex-pair}, \ref{thm-multiple-trav}, and \ref{thm-trav-general}
and the form of the solution at main order.
In Section \ref{sec2}, we introduce the radial ground state solution to (\ref{plasma0}) and study the non-degeneracy of the linearized operator around it.
Section~\ref{sect-proof-thm-vortex-pair} is devoted to the proof of Theorem~\ref{thm-vortex-pair},
with some some arguments deferred for later: a solvability theory for the linearized equation is developed in section~\ref{sect-linear-th} and some computations associated to the nonlinear problem are in section~\ref{sect-nl}.
We give some ideas of the proofs of Theorems
\ref{thm-rot}--\ref{thm-rot2} in section~\ref{sect-extensions}.

\section{An elliptic equation for concentrated solutions of \texorpdfstring{\eqref{orthogonal}}{o}}
\label{sect-scheme}

To prove Theorems~\ref{thm-vortex-pair}, \ref{thm-multiple-trav} and \ref{thm-trav-general} we need to find a family of solutions  $ \Theta_\ve(x)$ to the equation
\begin{align}
\label{bb}
(\nabla^\perp \Psi-ce_2) \cdot  \nabla \Theta =0	,
\quad \Psi = (-\Delta)^{-s} \Theta ,
\end{align}
such that
$$
\Theta_\ve(x) \rightharpoonup \sum_{j=1}^k m_j\delta (x-b_j^0) ,
$$
for given intensities $m_j$ and a solution ${\bf b}^0= (b_1^0,\ldots, b_k^0)$ of system \equ{reduced-trav}.
In order to achieve this we consider the following elliptic problem.
\be\label{eqtrans}
\left\{
\begin{aligned}
(-\Delta)^s \psi &=
\varepsilon^{(2-2s)\gamma-2}
\sum_{j=1}^k \sigma_j( \sigma_j(\psi + c x_1) - \varepsilon^{2s-2} \lambda_j)_+^\gamma \chi_{B_\delta(b_j)}
\quad \text{in }\R^2,
\\
\psi(x) & \to 0 \quad \text{as }|x|\to \infty,
\end{aligned}
\right.
\ee
where we take
$\sigma_j= +1$  if  $m_j>0$ and $=-1$ if $m_j<0$. The scalars  $\lambda_j$ will be suitably chosen later on. We also assume that $1<\gamma<\frac{2+2s}{2-2s}$.  The number $\delta>0$ is fixed so that the balls $B_\delta(b_j)$ are disjoint and the points $b_1,\ldots,b_k$ are close to $b_1^0,\ldots,b_k^0$.

\medskip
We look for sufficiently smooth solutions $\psi$  of problem \equ{eqtrans} such that
$$
\Theta_\varepsilon(x) :=   \varepsilon^{(2-2s)\gamma-2}
\sum_{j=1}^k \sigma_j( \sigma_j(\psi + c x_1) - \varepsilon^{2-2s} \lambda_j)_+^\gamma \chi_{B_\delta(b_j)}(x)
$$
has its support  contained in $\bigcup_{j=1}^k B_\delta (b_j)$.
We readily check that the latter condition guarantees that $\Theta_\ve(x)$ solves \equ{bb}.

\medskip
We will find a solution of problem \equ{eqtrans}
which at main order looks like a superposition
of sharply scaled similar radial profiles centered near each of the points $b_j$.
In \cite{chan_uniqueness_2020} it was proved that there exists a unique radial solution $W(y)$ of the problem
\begin{align}
\label{plasma0-2}
\left\{
\begin{aligned}
(-\Delta)^s W &=(W - 1)_+^\gamma \quad \mbox{ in }\R^2	,
\\
W(y)& \to 0 \quad \text{as }|y|\to \infty ,
\end{aligned}
\right.
\end{align}
where $0<s<1$ and $1<\gamma<\frac{2+2s}{2-2s}$. This solution has the precise asymptotic behavior
\begin{equation}
\label{asympt}
W(y)=M_\gamma  c_{2,s}|y|^{-(2-2s)} (1+ o(1)) \quad \mbox{ as }|y|\to \infty,	
\end{equation}
where $M_\gamma =\int_{\R^n}(W-1)_+^\gamma dy>0$ and $c_{2,s}$ is given in \eqref{cns}.

We look for a solution of \equ{eqtrans}
that looks approximately like
\begin{align}
\label{general-ansatz}
\psi_0(x) =
\varepsilon^{2s-2}
\sum_{j=1}^k
\sigma_j \mu_j^{-\frac{2s}{\gamma-1}}
W\Big( \frac{x-b_j}{\varepsilon\mu_j}\Big),
\end{align}
where $\mu_j$ are positive constants.
Then  as $\varepsilon\to 0$ we have
\begin{align*}
(-\Delta)^s \psi_0 (x)
\rightharpoonup M_\gamma \sum_{j=1}^k
\sigma_j \mu_j^{2(1-\frac{s\gamma}{\gamma-1})} \delta(x-b_j).
\end{align*}
Therefore we fix $\mu_j>0$ such that
\begin{align}
\label{mj}
M_\gamma \sigma_j \mu_j^{2(1-\frac{s\gamma}{\gamma-1})} = m_j, \quad j=1,\ldots, k,
\end{align}
which is possible if in addition  we assume that $\gamma \not=\frac{1}{1-s}$.

We compute, for $x\in B_\delta(b_l)$,
assuming for simplicity that $\sigma_l=1$,
\begin{align}
\label{error0}
& (-\Delta)^s \psi_0
-\varepsilon^{(2-2s)\gamma-2}
\sum_{j=1}^k \sigma_j( \sigma_j(\psi_0
 +c x_1) - \varepsilon^{2-2s} \lambda_j)_+^\gamma \chi_{B_\delta(b_j)}
\\
\nonumber
& \quad = \varepsilon^{-2} \mu_l^{-\frac{2s\gamma}{\gamma-1}}
\Big[
\Big( W\Big(\frac{x-b_l}{\varepsilon \mu_l}\Big) -1 \Big)_+^\gamma
\\
\nonumber
& \qquad
-
\Big(
W\Big(\frac{x-b_l}{\varepsilon \mu_l}\Big)
+
\sum_{j\not=l}
\sigma_j
\mu_l^{\frac{2s}{\gamma-1}} \mu_j^{-\frac{2s}{\gamma-1}}
W\Big(\frac{x-b_j}{\varepsilon \mu_j}\Big)
+ c \mu_l^{\frac{2s}{\gamma-1}} \varepsilon^{2-2s} x_1 - \mu_l^{\frac{2s}{\gamma-1}} \lambda_l
\Big)_+^\gamma
\Big] .
\end{align}
We note that $\psi_0$ is a good approximation to a solution to \eqref{eqtrans} if the parameters $\lambda_j $ are chosen such that
\begin{align}
\label{eq-lambda}
\sum_{j\not=l}
\sigma_j
\mu_l^{\frac{2s}{\gamma-1}} \mu_j^{-\frac{2s}{\gamma-1}}
W\Big(\frac{b_l-b_j}{\varepsilon \mu_j}\Big)
+ c \mu_l^{\frac{2s}{\gamma-1}} \varepsilon^{2-2s} b_{l,1} - \mu_l^{\frac{2s}{\gamma-1}} \lambda_l
=-1.
\end{align}
Similarly, if $\sigma_l=-1$ we impose that
\begin{align}
\label{eq-lambda2}
-\sum_{j\not=l}
\sigma_j
\mu_l^{\frac{2s}{\gamma-1}} \mu_j^{-\frac{2s}{\gamma-1}}
W\Big(\frac{b_l-b_j}{\varepsilon \mu_j}\Big)
- c \mu_l^{\frac{2s}{\gamma-1}} \varepsilon^{2-2s} b_{l,1}
- \mu_l^{\frac{2s}{\gamma-1}} \lambda_l
=-1.
\end{align}
Using the expansion \eqref{asympt} we get that
\[
\lambda_l = \mu_l^{-\frac{2s}{\gamma-1}} + O(\varepsilon^{2-2s})
\]
as $\varepsilon\to 0$.
With this choice of $\lambda_l$, we find that the error of approximation created by $\psi_0$, has the estimate
\begin{align}
\nonumber
&
\varepsilon^2
\Big\{
(-\Delta)^s \psi_0
-\varepsilon^{(2-2s)\gamma-2}
\sum_{j=1}^k \sigma_j( \sigma_j(\psi_0+c x_1) - \varepsilon^{2-2s} \lambda_j)_+^\gamma \chi_{B_\delta(b_j)}
\Big\}
\\
\label{error1}
& \quad = O(\varepsilon^{3-2s})
\sum_{j=1}^k\chi_{B_{C \varepsilon}(b_j)}.
\end{align}
The proof of Theorems
\ref{thm-vortex-pair}, \ref{thm-multiple-trav} and \eqref{thm-trav-general} consists in finding a solution of problem \equ{eqtrans} as a suitable
small perturbation of the function
$\psi_0$ defined above. Linearizing the problem around $\psi_0$ and a
Lyapunov-Schmidt reduction procedure, transforms the problem into one of adjusting the points $b_j$ as a small perturbation of a solution of system
\eqref{reduced-trav}. Nondegeneracy of the reduced limiting problem and that of the linearized elliptic operator play a crucial role in the complete proof. The proofs of Theorems \ref{thm-rot} and \ref{thm-rot2} follow from similar considerations.

\medskip
Of central importance will be the understanding of invertibility properties of
the linearized operator of equation \eqref{plasma0-2}, namely
\begin{align*}
\left\{
\begin{aligned}
L_0[\phi] & :=(-\Delta)^s\phi - \gamma (W(y) -1)_+^{\gamma-1} \phi  ,
\\
\phi(y) & \to 0 \quad \hbox{as } |y|\to \infty .
\end{aligned}
\right.
\end{align*}
We will prove in Section~\ref{sec2} that the only decaying elements in the kernel of $L_0$ are the ones associated to the invariance of \eqref{plasma0-2} under translations.

We will proceed with the detailed proof of Theorem~\ref{thm-vortex-pair} in section~\ref{sect-proof-thm-vortex-pair} below.

\section{The nondegeneracy of the building blocks}
\label{sec2}

Recently, in \cite{chan_uniqueness_2020}, the authors studied the ground state solution $w$ to the fractional plasma equation:
\begin{equation}\label{model}
(-\Delta)^s W=(W-1)_+^\gamma \quad \mbox{ in }\R^n	,
\end{equation}
where $s\in(0,1)$, $ 1\leq \gamma<\frac{n+2s}{n-2s}$, \, and they proved the following:
\begin{theorem}[Theorem 1.1 in \cite{chan_uniqueness_2020}]
\label{thm-ground-state}
There exists a unique ground state $W$ to equation (\ref{model}). Moreover, the ground state solution satisfies:
\begin{itemize}
\item [(i)] $W(x)=W (|x|)$ is radial symmetric and decreasing in $r=|x|$;
\item [(ii)] It satisfies the following asymptotic behavior:
\begin{equation}\label{asymptotic}
W(x)=M_\gamma c_{n,s}|x|^{-(n-2s)} (1+ o(1)) \quad \mbox{ as }|x|\to \infty,	
\end{equation}
where $M_\gamma=\int_{\R^n}(W-1)_+^\gamma dx>0$.
Moreover,
\begin{equation}\label{asymptotic1}
W'(|x|)=-(n-2s) M_\gamma c_{n,s}|x|^{-(n-2s)-1} (1+ o(1)) \quad \mbox{ as }|x|\to \infty	.
\end{equation}
\end{itemize}	
\end{theorem}

In this section, we will study the linearized operator of this equation around the ground state,
namely,
\begin{equation}\label{model-linear}
L_0[\phi]=(-\Delta)^s \phi-\gamma(W-1)_+^{\gamma-1}\phi.
\end{equation}
By invariance of (\ref{model}) under translations, every directional derivative of $W$ is in the kernel of $L_0$.
Equation \eqref{model} is also invariant under the dilations
\[ \lambda^{\frac{2s}{\gamma-1}} ( W( \lambda x) -1 )+1
\quad \text{where} \quad \lambda>0.\] Thus
\[
z_0(x) = \frac{d}{d \lambda}
\big[
\lambda^{\frac{2s}{\gamma-1}} \big( W( \lambda x) -1 \big) \big]_{\lambda=1}
\]
also satisfies $L_0[z_0]=0$. Note that $z_0\in L^\infty(\R^n)$, but $\lim_{|x|\to \infty} z_0(x) = -\frac{2s}{p-1} \not=0$.

The main result here is the following.

\begin{proposition}
\label{model-nondegeneracy}
If $\phi$ is in the kernel of $L_0$ and $\phi(x)\to0$ as $|x|\to \infty$,
then $\phi$ is a linear combination of
$\frac{\partial W}{\partial x_1}, \cdots, \frac{\partial W}{\partial x_n}$.
\end{proposition}

Before giving the proof we introduce some notation.
Let
\[
V(r)=\gamma (W(r)-1)_+^{\gamma-1} ,
\]
so that  $ L_0[\phi] = (-\Delta)^s \phi - V(r) \phi$.

We will use the  extension problem for the fractional Laplacian  \cite{caffarelli-silvestre-extension}, so that equation \eqref{model-linear} can be solved through the boundary reaction problem
\begin{align}
\label{eqlinear1}
\left\{
\begin{aligned}
\partial_{yy}\Phi+\dfrac{1-2s}{y}\partial_y \Phi+\Delta_{\R^n}\Phi
&=0  && \mbox{in }\R^{n+1}_+,
\\
-\tilde d_s\lim_{y\to0} y^{1-2s}\partial_y \Phi
&=V(r) \phi &&\mbox{on }\R^n ,
\end{aligned}
\right.
\end{align}
where $\tilde d_s=-\frac{2^{2s-1}\Gamma(s)}{s\Gamma(-s)}$ and $\Phi|_{y=0}=\phi$. Were are using the notation
$(x,y) \in \R^{n+1}$, $x\in \R^n$, $y\in \R$ and $\R^{n+1}_+ = \{ (x,y) \in \R^{n+1} | y>0\}$.

Le us consider the spherical harmonic decomposition of $\mathbb S^{n-1}$. By $\mu_i$, $i=0,1,\ldots$ we denote the $i$-th eigenvalue of $-\Delta_{\mathbb S^{n-1}}$, repeated according to multiplicity and arranged in increasing order, and by $E_i(\theta)$ the corresponding eigenfunction.
In particular $E_0$ is constant and $E_1,\ldots,E_n$ are the coordinate functions $x_i$ normalized.
Then we can write
\begin{align}
\label{F}
\Phi(x,y)=\sum_{i=0}^\infty
\Phi_i(x,y), \quad
\Phi_i(x,y) = \psi_i(r,y)E_i(\theta),
\end{align}
where  $x = r \theta$, $r=|x|$, $\theta \in \mathbb S^{n-1}$, $y>0$,
so that $\Phi_i$  is the Fourier mode $i$ of $\Phi$.
We note that  $L_0[\phi_i]=0$ where $\phi_i(x) = \Phi_i(x,0)$.
We have also that $\psi_i$ satisfies:
\begin{align}
\label{eqlinear2}
\left\{
\begin{aligned}
\partial_{yy}\psi_{i}+\dfrac{1-2s}{y}\partial_y\psi_{i}
+ \partial_{rr}\psi_i
+ \frac{n-1}{r}\partial_r \psi_i -\frac{\mu_i}{r^2}\psi_i
&=0 &&\mbox{in }\R^{n+1}_+,
\\
-\tilde d_s\lim\limits_{y\to0} y^{1-2s}\partial_y \psi_i
&=V(r)\psi_i
&&\mbox{on }\R^n .
\end{aligned}
\right.
\end{align}

\begin{proof}[Proof of Proposition~\ref{model-nondegeneracy}]


Let $\phi\in L^\infty(\R^n)$ satisfy $L_0[\phi]=0$ and suppose also that
$ \phi(x) \to 0$ as $|x|\to\infty$.
We claim that
\begin{align}
\label{newtonian2}
\phi(x)
=(-\Delta)^{-s}[\gamma(\phi-1)_+^{\gamma-1}\phi](x)	=c_{n,s}\int_{\R^n}
\frac{\gamma(w(y)-1)_+^{\gamma-1}}{|x-y|^{n-2s}}\phi(y)\,dy.
\end{align}
Indeed, let
\[
\tilde \phi(x) = c_{n,s}\int_{\R^n}
\frac{\gamma(w(y)-1)_+^{\gamma-1}}{|x-y|^{n-2s}}\phi(y)\,dy ,
\]
and observe that
\[
|\tilde \phi(x)|\leq C(1+|x|)^{-(n-2s)}.
\]
We have that $\phi - \tilde \phi$ is $s$-harmonic, and by the result of \cite{fall}, it is an affine function. Since $\phi(x) - \tilde \phi(x) \to 0$ as $|x|\to \infty$ we deduce \eqref{newtonian2}.

\medskip\noindent
{\bf Step 1.} First we consider the mode zero case, that is, $\phi_0$ with the notation \eqref{F},  which is radial element of the kernel of $L_0$.

We claim that $\phi_0 = a z_0$ for some constant $a\in \R$.
Indeed, consider the function $u:=\phi_0 - a z_0$ where $a$ is chosen so that  $u(0)=0$.
We note that $L_0[u]=0$ and that $u\in L^\infty(\R^n)$.
We will use the argument in \cite{frank_uniqueness_2016} to prove that $u\equiv 0$.

The ground state $W$ of \eqref{model} constructed in Theorem~\ref{thm-ground-state} satisfies that there exists a unique $R_0>0$ such that $W(R_0)=0$.
By standard estimates for the extension problem \eqref{eqlinear1}, $W(r)$ is smooth for $r\not=R_0$
(see for example the appendix in \cite{frank_uniqueness_2016}
and \cite{cabre-sire}).
Hence $V(r)$ is smooth for $r\not=R_0$.
Let $U$ denote the extension of $u$ to $\R^{n+1}_+$ solving \eqref{eqlinear1}.
Following \cite{frank_uniqueness_2016,cabre-sire},
for $r>0$ let
\[
H(r) = \tilde d_{s}
\int_0^\infty \frac{y^{1-2s}}{2}
[ (\partial_r U(r,y) )^2
-(\partial_t U(r,y) )^2]\,dy
+ \frac{1}{2}V(r) u(r)^2 .
\]
We note that $H(r)$ is well defined and continuous for all $r\geq 0$ and smooth for $r\not=R_0$.
We also observe that $H(r)\to 0$ as $r\to \infty$, since $V(r)=0$ for $r>R$.
As in \cite{frank_uniqueness_2016} we have
\[
H'(r) = -\tilde d_{s}
\frac{n-1}{r}
\int_0^\infty y^{1-2s}
(\partial_r U(r,y) )^2
\,dy
+ \frac{1}{2}V'(r) u(r)^2  \leq 0 \quad\text{for } r>0, \ r\not=R_0.
\]
Since $H(0)\leq 0$ and $H(r)\to 0$ as $r\to \infty$ we conclude that $H(r) = 0$ for all $r>0$ and therefore $u$ must be constant. Since $u(0)=0$, we find that $u\equiv 0$, that is, $\phi_0 = a z_0$.
But since $\lim_{r\to\infty}z_0(r) \not=0$
and $\lim_{r\to\infty}\phi_0(r) =0$, we deduce that $a=0$, so $\phi_0 \equiv 0$.

\medskip\noindent
{\bf Step 2.} The modes $1,\ldots,n$.
We consider $\phi_i$, $i=1,\ldots,n$ with the notation \eqref{F}, which are elements in the kernel of $L_0$ in Fourier mode $i$. We denote by $\Phi_i$ the extension of $\phi_i$ and write $\Phi_i(x,y) = \psi_i(r,y) E_i(\theta)$ as in \eqref{F}.

Differentiating the equation \eqref{model}
we get
\[
L_0[z_i]=0, \quad z_i = - \frac{\partial W}{\partial x_i}.
\]
Let $Z_i$ be the extension of of $z_i$ solving \eqref{eqlinear1} and write
$ Z_i = Z_*(r,y) E_i(\theta)$ so that $Z_*$ solves \eqref{eqlinear2}.
Note that $z_i = W'(r) \frac{x_i}{|x|}$ and that $W'(r)<0$ by the results in \cite{chan_uniqueness_2020}.
By the strong maximum principle $Z_*(r,y)>0$ for all $r>0$ and all $y\geq 0$.
We consider the function $ \varphi = \frac{\psi_i}{Z_*}$ and note that it satisfies
\[
\frac{1}{r^{n-1}} \partial_r ( r^{n-1} Z_*^2 \partial_r \varphi)
+ \frac{1}{y^{1-2s}} \partial_y ( y^{1-2s} Z_*^2 \partial_y \varphi) = 0, \quad r>0, \quad y>0.
\]
We multiply this equation by $\varphi r^{n-1} y^{1-2s}$ and integrate in the region $D_{\epsilon,R} = \{(r,y)| r>0, \ y>0, \ \epsilon^2 < r^2+y^2 < R^2\}$ where $0<\epsilon< R$, to find that
\begin{align*}
\int_{D_{\epsilon,R}} Z_*^2 [ (\partial_r \varphi)^2
+ (\partial_y \varphi)^2] r^{n-1} y^{1-2s} \, d r d y
= - I_{\epsilon} + I_{R}
\end{align*}
where
\begin{align*}
I_{\rho} = \int_{\mathcal C_\rho}
Z_*^2
[ r \varphi \partial_r \varphi
+ y \varphi \partial_y \varphi ]
\frac{r^{n-1} y^{1-2s}}{(r^2+y^2)^{\frac{1}{2}}}
\,d \ell ,
\end{align*}
where $ \mathcal C_\rho = \{(r,y)|r>0,y>0,r^2+y^2=\rho^2\}$
and $\ell$ is archlength in $(r,y)$.
We claim that
\begin{align}
\label{limits}
\lim_{R\to \infty} I_R  = 0 ,
\quad
\lim_{\epsilon\to0} I_\epsilon  = 0.
\end{align}
To prove these statements, we note first that
\[
Z_*^2 \varphi \partial_r \varphi =
\psi_i \partial_r \psi_i - \frac{\psi_i}{Z^*} \psi_i \partial_r Z_* ,
\quad
Z_*^2 \varphi \partial_y \varphi =
\psi_i \partial_y \psi_i - \frac{\psi_i}{Z^*} \psi_i \partial_y Z_* .
\]
We recall some estimates for  first $R$
we observe that using the Poisson kernel for the operator $\Delta_x + \frac{1-2s{y}} \partial_y$ \cite{caffarelli-silvestre-extension} we have that
\[
Z_* (r,y)\geq c\frac{y^{2s}}{(r^2+y^2)^{\frac{n+2s}{2}}}, \quad
|\psi_i(r,y)| \leq C \frac{y^{2s}}{(r^2+y^2)^{\frac{n+2s}{2}}} ,
\]
for $r^2+y^2$ large, where $c>0$ is a constant. Using these estimates with similar ones for the derivatives we find that for $R$ large
\begin{align*}
|I_R| \leq C R^{-n-1-2s} ,
\end{align*}
where $C$ is a constant. This proves the first limit in \eqref{limits}.

For the estimate of $I_\epsilon$ when $\epsilon>0$ is small, we first observe that
\[
Z_*(r,y) \geq c r
\]
for $r^2+y^2<\delta^2$ and some $c>0$, $\delta>0$. This is proved using the maximum principle applied to the equation \eqref{eqlinear2} with the subsolution $r+b r^2 $ where $b>0$. We also have the estimates
\[
y^{1-2s} |\partial_y \psi_i|
+
y^{1-2s} |\partial_y Z_*|
\leq C
\]
for $r^2+y^2<\delta^2$, which follow from \cite{cabre-sire}[Lemma 4.5].
Then
\begin{align*}
|I_\epsilon|\leq C \epsilon^{n-2s} \to 0
\end{align*}
as $\epsilon\to0$.
This proves the second limit in \eqref{limits}.

Using \eqref{limits} we deduce that
\[
\int_{\{r>0,\,y>0\}} Z_*^2 [ (\partial_r \varphi)^2
+ (\partial_y \varphi)^2] r^{n-1} y^{1-2s} \, d r d y =0,
\]
which implies that $\varphi$ is constant. We deduce then that  $ \phi_i = c_i z_i$ for some constant $c_i$.

\medskip
\noindent
{\bf Step 3.} The remaining modes $m\geq{n+1}$. We use an integral estimate  as in \cite{davila_fast_2008, davila_supercritical_2007}.

As before, let $z_*(r) = - W'(r)>0$ and $Z_*(r,y)$ be its extension.
Consider $\phi_m$ with $m\geq n+1$ and $\Phi_m$ its extension, which we write as $\Phi_m(x,y) = \psi_m(r,y) E_m(\theta)$.

Let us rewrite \eqref{eqlinear2} as
\[\left\{
\begin{aligned}
div(y^{1-2s}\nabla\psi_m)
&=\mu_{m}\dfrac{y^{1-2s}}{r^2}\psi_m
    &&\text{in } \R^{n+1}_+,\medskip\\
-\tilde{d}_s\lim\limits_{y\to0}y^{1-2s}\p_{y}\psi_m
    &= V(r) \psi_m
    &&\text {on } \partial \R^{n+1}_+.
\end{aligned}
\right.\]
We multiply this equation by $Z_*1$ and the one with $m=1$ by $\psi_m$. Their difference gives the equality
\begin{align*}
(\mu_m-\mu_1)\frac{y^{1-2s}}{r^2}\psi_m Z_*
&=Z_* div(y^{1-2s}\nabla\psi_m)-\psi_m div(y^{1-2s}\nabla Z_*)\\
&=div(y^{1-2s}(Z_*\nabla\psi_m-\psi_m\nabla Z_*)).
\end{align*}

Let us integrate over the region where $\psi_m>0$.  The boundary $\partial\{\psi_m>0\}$ is decomposed into a disjoint union of $\partial^0\{\psi_m>0\}$ and $\partial^+\{\psi_m>0\}$, on which the extension variable $y=0$ and $y>0$, respectively.
Hence
\begin{align*}
0&\leq\tilde{d}_s(\mu_m-\mu_1)
\int_{\{\psi_m>0\}}\frac{\psi_m Z_*}{r^2}\,dxdy
\\
&=\int_{\partial^0\{\psi_m>0\}}\left( Z_* \lim_{y\to 0}y^{1-2s}\frac{\partial \psi_m}{\partial \nu}-\psi_m\lim_{y\to 0}y^{1-2s}\frac{\partial Z_*}{\partial\nu} \right)\,dx\\
& \quad +\int_{\partial^+\{\psi_m>0\}}y^{1-2s}\left(Z_*\frac{\partial \psi_m}{\partial \nu}-\psi_m\frac{\partial Z_*}{\partial \nu}\right)\,dxdy.
\end{align*}

The first integral on the right hand side vanishes due to the equations satisfied by $Z_*$ and $\psi_m$. Then we observe that on $\partial^+\{\psi_m>0\}$, one has $Z_*>0$, and $\psi_m=0$. This forces (using $\mu_m>\mu_1$)
\[\int_{\{\psi_m>0\}}\frac{\psi_m Z_*}{r^2}\,dxdy=0,\] which in turn implies $\psi_m\leq0$. Similarly $\psi_m\geq0$ and, therefore, $\psi_m\equiv0$ for $m\geq{n+1}$. This completes the proof of the Proposition \ref{model-nondegeneracy}.
	
\end{proof}

A different proof can be done using the ODE techniques for fractional problems as in \cite{ao_ode-methods_2020}.

\hide{

Let $\lambda $ be a fixed positive constant and consider now the problem
\begin{equation}\label{newlimit}
(-\Delta )^s u - (u-\lambda )_+^\gamma = 0 , \quad {\mbox {in}} \quad \R^n
\end{equation}
with $u(x) \to 0$ as $|x| \to \infty$, where $s\in(0,1)$, $ 1\leq \gamma<\frac{n+2s}{n-2s}$. This problem has a unique radial solution given by
\begin{equation}\label{wla}
w^\lambda (x) = \lambda w (\lambda^{\frac{\gamma-1}{2s}} x).
\end{equation}
The linearized operator of this equation around the $w^\lambda$ is given by
\begin{equation}\label{model-linear}
L_\lambda [\phi]=(-\Delta)^s \phi-\gamma(w^\lambda -\lambda )_+^{\gamma-1}\phi.
\end{equation}
By invariance of (\ref{model}) under translations, every directional derivative of $w$ is in the kernel of $L_0$.
Problem \eqref{model} is also invariant under the dilations
$ \lambda^{\frac{2s}{\gamma-1}} ( w( \lambda x) -1 )+1$ where $\lambda>0$. Thus
\[
z_0(x) = \frac{d}{d \lambda}
\big[
\lambda^{\frac{2s}{\gamma-1}} \big( w( \lambda x) -1 \big) \big]_{\lambda=1}
\]
also satisfies $L_0[z_0]=0$. Note that $z_0\in L^\infty(\R^n)$, but $\lim_{|x|\to \infty} z_0(x) = -\frac{2s}{p-1} \not=0$.

The main result here is the following.

\begin{proposition}
\label{model-nondegeneracy}
If $\phi$ is in the kernel of $L_0$ and $\phi(x)\to0$ as $|x|\to \infty$,
then $\phi$ is a linear combination of
$\frac{\partial w}{\partial x_1}, \cdots, \frac{\partial w}{\partial x_n}$.
\end{proposition}

}

\section{Construction of a vortex-anti vortex pair: the proof of Theorem~\ref{thm-vortex-pair}}
\label{sect-proof-thm-vortex-pair}

Let us consider the traveling vortex-anti vortex pair described by \eqref{trav-param}, centered at $b_1 = (d,0)$, $b_2=(-d,0)$ with masses $m$ and $-m$.
For the construction of the solution stated in Theorem~\ref{thm-vortex-pair} we take the formulation described in Section~\ref{sect-scheme} and specialize it
to
the following problem
\begin{align}
\label{eq00}
\left\{
\begin{aligned}
(-\Delta)^s \psi &=
\varepsilon^{(2-2s)\gamma-2}
\big( (\psi + c x_1 - \varepsilon^{2-2s} \lambda)_+ ^\gamma \chi_{B_\delta(b_1)}
-(-\psi - c x_1 -\varepsilon^{2-2s} \lambda)_+ ^\gamma \chi_{B_\delta(b_2)}
\big)
\\
& \qquad \text{in }\R^2 ,
\\
\psi(x) & \to 0 \quad \text{as }|x|\to \infty ,
\end{aligned}
\right.
\end{align}
where  we have taken the same $\lambda$ for both points.

\medskip
Given the symmetries of the problem, it is natural to construct a solution with
\begin{align}
\label{symmetry2}
\psi(x_1,x_2)=-\psi(-x_1,x_2)=\psi(x_1,-x_2)
\quad \text{for all }(x_1,x_2)\in \R^2 .
\end{align}
Following \eqref{general-ansatz} we take as a first approximation of a solution to \eqref{eq00}
\begin{align}
\label{psi01}
\psi_0(x) = \varepsilon^{2s-2}  \mu^{-\frac{2s}{\gamma-1}}
W\Big( \frac{x-b_1}{\varepsilon \mu} \Big)
-  \varepsilon^{2s-2}  \mu^{-\frac{2s}{\gamma-1}}
W\Big( \frac{x-b_2}{\varepsilon \mu} \Big) ,
\end{align}
where $W$ is the ground state of \eqref{plasma0-2}
and $\mu>0$ is defined as in relation
\equ{mj} by
\begin{align}
\label{masses}
M_\gamma  \mu^{2(1-\frac{s\gamma}{\gamma-1})} = m.
\end{align}

\medskip
Let us write
\begin{align*}
S(\psi)
&=
(-\Delta)^s \psi -
\varepsilon^{(2-2s)\gamma-2}
\big[ (\psi + c x_1 - \varepsilon^{2-2s} \lambda)_+ ^\gamma \chi_{B_\delta(b_1)}
\\
& \qquad \qquad
-(-\psi - c x_1 -\varepsilon^{2-2s} \lambda)_+ ^\gamma \chi_{B_\delta(b_2)}
\big] ,
\end{align*}
so that problem \eqref{eq00} is equivalent to $S(\psi)=0$ and $\psi(x)\to 0$ as $|x|\to \infty$.
We recall that the scalar $\lambda$ is defined by \eqref{eq-lambda},
which in the case of a vortex-anti vortex pair becomes
\begin{align}
\label{lambda2}
W\Big( \frac{b_1-b_2}{\varepsilon\mu}  \Big)
- c \varepsilon^{2-2s} \mu^{\frac{2s}{\gamma-1}} d
+  \mu^{\frac{2s}{\gamma-1}} \lambda
= 1  ,
\end{align}
since $b_1=(d,0)$, $b_2=(-d,0)$.
This gives $\lambda = \lambda(c,d) $ with the form
\[
\lambda = \mu^{-\frac{2s}{\gamma-1}} + O( \varepsilon^{2-2s}) >0.
\]
With this choice of $\lambda$, we obtain
from \eqref{error1} that for $y \in B_{\delta/(\mu \varepsilon)}(b_1')$
\begin{align}
\nonumber
S(\psi_0)
& = \varepsilon^{-2} \mu^{-\frac{2s\gamma}{\gamma-1}}
\Big[
(W(y-b_1')-1)_+^\gamma
\\
\nonumber
& \qquad \qquad \qquad
- \big( W(y-b_1') -1 + W (b_1'-b_2')- W(y-b_2')
\\
\label{error}
& \qquad \qquad \qquad
+ c \varepsilon^{3-2s} \mu^{\frac{2s}{\gamma-1}+1} (y_1-d)
\big)_+^\gamma \,
\chi_{ \{ y \in B_{\delta/(\varepsilon\mu)}(0) \} }
\Big] ,
\end{align}
where
\[
y = \frac{x}{\varepsilon\mu} = (y_1,y_2)\in \R^2,
\quad
b_j' = \frac{b_j}{\varepsilon\mu}.
\]
We will work with  the parameters $c$, $\mu>0$ fixed  and
\begin{align}
\label{condition-d}
d \in \Big(d_0,\frac{1}{d_0} \Big)
\end{align}
to be adjusted, where $d_0>0$ is fixed small. Then we see that
\[
S(\psi_0) = O \big(\varepsilon^{-2} \mu^{-\frac{2s\gamma}{\gamma-1}} \varepsilon^{3-2s} \chi_{B_{C \varepsilon} (0)} \big) ,
\]
for some constant $C>0$.

It will be convenient to work with the unknown $v$ defined by
\[
\psi(x) = \varepsilon^{2s-2} \mu^{-\frac{2s}{\gamma-1}}
v \Big( \frac{y}{\varepsilon\mu} \Big).
\]
We note that
\begin{align*}
S(\psi) = \varepsilon^{-2} \mu^{-\frac{2s\gamma}{\gamma-1}}
\big[ (-\Delta_y)^s v - f(y,v) \big],
\end{align*}
where
\begin{align*}
f(y,v) &= \big( v
+ c \varepsilon^{3-2s} \mu^{\frac{2s}{\gamma-1}+1} y_1
- \mu^{\frac{2s}{\gamma-1}} \lambda \big)_+^\gamma
\chi_{B_{\delta/(\varepsilon\mu)}(b_1')}
\\
& \quad
- \big( -v - c \varepsilon^{3-2s} \mu^{\frac{2s}{\gamma-1}+1} y_1
- \mu^{\frac{2s}{\gamma-1}} \lambda \big)_+^\gamma
\chi_{B_{\delta/(\varepsilon\mu)}(b_2')} .
\end{align*}

Thus  \eqref{eq00} becomes the nonlinear problem
\begin{align}
\label{eq01}
\left\{
\begin{aligned}
(-\Delta)^s v &= f(y,v) \quad \text{in }\R^2
\\
v(y) & \to 0 \quad \text{as }|y|\to \infty,
\end{aligned}
\right.
\end{align}
and the ansatz \eqref{psi01} takes the form
\begin{align}
\label{def-v0}
v_0(y) = W(y-b_1') - W(y-b_2').
\end{align}
We look for a solution of \eqref{eq01} of the form $v=v_0 + \phi$.
Then equation \eqref{eq01} is equivalent to
\begin{align}
\label{nonlinear1}
L[\phi] = - E + N[\phi]
\end{align}
where
\begin{align}
\nonumber
L[\phi] &= (-\Delta)^s \phi
- V(y)\phi , \quad V(y) = f_v(y,v_0)
\\
\nonumber
E &=  (-\Delta_y)^s v_0 - f(y,v_0)
\\
\label{def-N}
N[\phi] &= f(y,v_0+\phi) - f(y,v_0) - f_v(y,v_0)\phi.
\end{align}



The symmetries \eqref{symmetry2} allow us to invert the operator $L$ up to one parameter, associated with the function
\begin{align*}
Z=Z_1+Z_2, \quad
Z_1(y)=\frac{\partial W}{\partial y_1}(y-b_1'), \quad  Z_2(y)=\frac{\partial W}{\partial y_1}(y-b_2').
\end{align*}
Note that the function $Z$ is odd in $y_1$ and even  in $y_2$.
Let us consider the projected linear problem
\begin{align}
\label{linearproblem0}
\left\{
\begin{aligned}
& L[\phi]=h(y)+ a V(y) Z(y), \quad \text{in }\R^2 ,
\\
& \int_{\R^2} V Z \phi \, dy=0 ,
\\
& \phi(y) \to 0 \quad \text{as }|y|\to \infty.
\end{aligned}
\right.
\end{align}
%
%
We introduce the following norms to deal with the
linear problem \eqref{linearproblem0}:
\begin{align*}
\|\phi\|_* &= \sup_{y\in \R^2} \rho(y)^{-(2-2s)} |\phi(y)| ,
\qquad
\|h\|_{**} = \sup_{y\in \R^2} \rho(y)^{-(2+\sigma)} |h(y)| ,
\end{align*}
where $0<\sigma<1$ and
\[
\rho(y) = \frac{1}{1+|y-b_1'|} + \frac{1}{1+|y-b_2'|} .
\]

We have the following:
\begin{proposition}
\label{existenceoflinearproblem0}
Assume that $d$ satisfies \eqref{condition-d} and $h$ satisfies $\|h\|_{**}<\infty$ and  the symmetries \eqref{symmetry2}. Then for $\varepsilon>0$ small there exists a unique solution $\phi=T_d(h)$ of (\ref{linearproblem0}), which defines a linear operator of $h$ and there exists $C>0$ independent of $\ve $ such that
$$
\|\phi\|_*+|a|\leq C\|h\|_{**}.	
$$
Moreover $\phi$ satisfies the symmetries \eqref{symmetry2}.
\end{proposition}
We prove this proposition in Section~\ref{sect-linear-th}.

\medskip

Instead of solving  problem \eqref{nonlinear1} directly, we consider the nonlinear projected problem
\begin{align}
\label{nonlinear}
\left\{
\begin{aligned}
& L[\phi]=-E+N[\phi] + a V Z, \quad \text{in }\R^2 \\
& \int_{\R^2} V Z \phi \, dy=0 ,
\\
& \phi(y) \to 0 \quad \text{as }|y|\to \infty .
\end{aligned}
\right.
\end{align}

\begin{proposition}
\label{existenceofnonlinearproblem}
Assume that $d$ satisfies \eqref{condition-d}.
There is $r_0>0$ such that for $\varepsilon>0$ small there exists a unique solution $\phi = \phi_d$  to (\ref{nonlinear})
in the ball $\|\phi\|_* \leq r_0$.
Moreover it satisfies
\begin{equation*}
\|\phi_d\|_*\leq C \varepsilon^{3-2s}
\end{equation*} 	
and $\phi_d$ is continuous with respect to $d \in (d_0,1/d_0)$.
\end{proposition}

The proof of this Proposition is in section~\ref{sect-nl}.

\medskip
We have obtained a solution $v_d=v_0+\varphi$
of
\begin{align}
\label{nl-p}
\left\{
\begin{aligned}
(-\Delta)^s v_d &= f(y,v_d) + a_d V Z,\quad \text{in }\R^2
\\
v_d(y) & \to 0 \quad \text{as }|y|\to \infty,
\end{aligned}
\right.
\end{align}
for some parameter $a_d$.
The fully solvability of \eqref{eq01} is reduced to finding $d$ such that  $a_d=0$.

Multiplying \eqref{nl-p} by $Z$ and integrating over $\R^2$, we have that
\begin{equation*}
a_d \int_{\R^2} V Z^2dx=	\int_{\R^2}[(-\Delta)^s v_d -f(y,v_d)]Zdy.
\end{equation*}

Thus  $a_d=0$ is reduced to
\begin{equation*}
\int_{\R^2}[(-\Delta)^s v_d-f(y,v_d)]\, Z\, dy=0.
\end{equation*}

We have the following:

\begin{proposition}
\label{reduedproblem}
If $d$ satisfies \eqref{condition-d}, then

\begin{equation*}
\int_{\R^2}[(-\Delta)^s v_d -f(y,v_d)]\, Z\, dy	=
c_0 \varepsilon^{3-2s}\Big(
\frac{1}{d^{3-2s}} +  c_1
\Big)
+ O(\varepsilon^{(3-2s) \min(\gamma,2)})
+ O(\varepsilon^{4-2s}),
\end{equation*}
where $c_0\not=0$ is a  constant and $c_1 = \frac{c}{m} \frac{4\pi \Gamma(s)}{\Gamma(2-s)}$ (here $c$ and $m$ are the parameters in \eqref{trav-param}).
\end{proposition}

The proof of this proposition is in section~\ref{sect-nl}.

\medskip

\noindent

\begin{proof}[Proof of Theorem~\ref{thm-vortex-pair}]
The equation $a_d=0$ is reduced to
\begin{equation*}
\frac{1}{d^{3-2s}} -  c_1
+ g(d)=0,
\end{equation*}
where $g$ is continuous and $g(d) = O(\varepsilon^{(3-2s) \min(\gamma,2)})
+ O(\varepsilon^{4-2s}) $ as $\varepsilon\to 0$.
Therefore we can find a solution $d=c_1^{-\frac{1}{3-2s}}(1+o(1))$.
\end{proof}

\section{Linear theory}
\label{sect-linear-th}
Here we prove Proposition~\ref{existenceoflinearproblem0}
but before  we obtain an a-priori estimate.

\begin{proposition}
\label{aprioriestimate0}
Assume that $d$ satisfies \eqref{condition-d} and $h$ satisfies $\|h\|_{**}<\infty$ and  the symmetries \eqref{symmetry2}. Then there is $C$ such that for $\varepsilon>0$ small and
for any solution $(\phi,a)$ of (\ref{linearproblem0})
satisfying the symmetries \eqref{symmetry2},
\begin{equation}\label{apriori1}
\|\phi\|_*+|a|\leq C\|h\|_{**}	.
\end{equation}	
\end{proposition}

\begin{proof}

First we claim that
\begin{equation}\label{estimateofc}
|a|\leq C(\|h\|_{**}+o(1)\|\phi\|_*).
\end{equation}	
The coefficient $c$ is determined by
\begin{equation*}
a\int_{\R^2} V Z^2dy=\int_{\R^2}L[\phi] Z\, dy-\int_{\R^2}h Zdy
\end{equation*}	
where one has
\begin{align*}
\int_{\R^2}VZ^2dy&=2\int_{\R^2}\gamma (W-1)_+^{\gamma-1}(\frac{\partial W}{\partial y_1})^2dy(1+o(1))
\\
&:=c_0(1+o(1))>0.	
\end{align*}
Note  that
\begin{align*}
\int_{\R^2}(-\Delta)^s \phi \, Z\, dy&=\int_{\R^2}\phi\, (-\Delta)^s\, Z\, dy\\
&=\gamma
\int_{\R^2} \phi[(W(y-b_1')-1)_+^{\gamma-1}Z_1 + (w(y-b_2')-1)_+^{\gamma-1}Z_2]dy.
\end{align*}
But for $y_1\geq 0$,
\begin{align*}
| \gamma (W(y-b_1')&-1)^{\gamma-1}_+Z_1
+\gamma(W(y-b_2')-1)_+^{\gamma-1}Z_2- V |\\
&=|\gamma(W(y-b_1')-1)^{\gamma-1}_+Z_1
- \gamma ( W(y-b_1) -1 + O(\varepsilon^{3-2s})
)_+^{\gamma-1}|
\\
& \leq
C \varepsilon^{(3-2s)\min(\gamma-1,1)}
\chi_{B_{R_0}(b_1')}
\end{align*}
where  $R_0$ is a fixed constant.
So
\begin{align*}
|a|&\leq C[O(\varepsilon^{(3-2s)\min(\gamma-1,1)} )\|\phi\|_*+\|h\|_{**}]
\end{align*}
and \eqref{estimateofc} follows.

Next we claim that
\begin{equation}\label{apriori}
\|\phi\|_*\leq C\|h\|_{**}.
\end{equation}
We argue by contradiction, assuming that there exist
$\ve_n\to 0$, $(\phi_n, a_n)$ solution  to (\ref{linearproblem0}) for some $h_n$ and $b_{1,n} = (d_n,0)$, $b_{2,n} = (-d_n,0)$ with $d_n$ satisfying \eqref{condition-d}, such that
\begin{equation}\label{assumption}
\|\phi_n\|_{*}=1, \ \|h_n\|_{**}\to 0 \quad \mbox{ as }n\to \infty,
\end{equation}	
and such that $\phi_n$, $h_n$ satisfy the symmetries \eqref{symmetry2}.
First we show that for any fixed $R_1>0$, we have that
\begin{equation}\label{trivial}
\|\phi_n\|_{L^\infty(B_{R_1}(b_{1,n}'))}+\|\phi_n\|_{L^\infty(B_{R_1}(b_{2,n}'))}\to 0,
\end{equation}	
where $b_{j,n}' = \frac{b_{j,n}}{\varepsilon\mu}$.
Indeed assume that for a subsequence $\|\phi_n\|_{L^\infty(B_{R_1}(b_{1,n}'))}\geq \gamma_0 >0$.
Let us set
\begin{equation*}
\bar{\phi}_n(y)=\phi_n(y+b_{1,n}').
\end{equation*}	
Then $\bar{\phi}_n$ satisfies
\begin{equation*}
(-\Delta)^s\bar{\phi}_n-V(y+b_{1,n}')\bar{\phi}_n=h_n(y+b_{1,n}')+a_n V(y+b_{1,n}')Z(y+b_{1,n}').	
\end{equation*}
One can rewrite it as
\begin{equation*}
(-\Delta)^s\bar{\phi}_n-\gamma (W-1)_+^{\gamma-1}\bar{\phi}_n=\bar{h}_n	
\end{equation*}
where
\begin{equation*}
\bar{h}_n(y)=h_n(y+b_{1,n}')+(V(y+b_{1,n}')-\gamma(W(y)-1)_+^{\gamma-1})\bar{\phi}_n+a_n V(y+b_{1,n}') Z(y+b_{1,n}').
\end{equation*}
Using that $V(y+b_{1,n}')-\gamma(W(y)-1)_+^{\gamma-1} = O (\varepsilon_n^{(3-2s)\min(\gamma-1,1)} ) $ and the assumptions
\eqref{assumption} we have $\bar{h}_n\to 0$ uniformly on compact sets. Thus passing to a sub-sequence, we may assume that $\bar{\phi}_n$ converge uniformly on compact sets
to a  function $\bar{\phi}$ with
\begin{equation}
\label{nontrivial}
\|\bar{\phi}\|_{L^\infty (B_{R_1}(0))}\geq \gamma_0.
\end{equation}
Furthermore, the limit function $\bar{\phi}$ is even in $y_2$, it solves
\begin{equation*}
(-\Delta)^s \bar{\phi}-\gamma (w-1)_+^{\gamma-1}\bar{\phi}=0,
\end{equation*}
with $\bar{\phi} (y) \to 0$ as $|y| \to \infty$
and it satisfies the orthogonality condition $\int_{\R^2}(W-1)_+^{\gamma-1}\bar{\phi}\frac{\partial W}{\partial y_1}dy=0$.
By the non-degeneracy Proposition \ref{model-nondegeneracy}, necessarily one has $\bar{\phi}=0$, which is a contradiction to (\ref{nontrivial}).
Next from the equation satisfied by $\phi_n$, one has
\begin{equation*}
(-\Delta)^s\phi_n=V \phi_n+h_n+a_n V Z.
\end{equation*}
Arguing as in Proposition~\ref{model-nondegeneracy} we have
\begin{equation*}
\phi_n(y)=c_{n,s}\int_{\R^2}
\frac{V\phi_n+h_n+a_n V Z}{|y-z|^{2-2s}}dz
\end{equation*}
and this implies that
\begin{equation*}
\rho(y)^{-(2-2s)}|\phi_n(y)|\leq C(\|h\|_{**}+|a_n|+\|\phi_n\|_{L^\infty (B_R(b_1')\cup B_R(b_2'))})	,
\end{equation*}
for some $R>0$ independent of $n $. Thus combining (\ref{trivial}), (\ref{assumption}) and (\ref{estimateofc}), one  gets that $\|\phi_n\|_*\to 0$, which is a contradiction to (\ref{trivial}).
So 	combining (\ref{estimateofc}) and (\ref{apriori}), one has (\ref{apriori1}).
\end{proof}

\begin{proof}[Proof of Proposition \ref{existenceoflinearproblem0}]
Consider
\begin{equation*}
H=\Big\{\phi\in  \mathring{H}^s(\R^2) \, | \,
\phi \text{ satisfies \eqref{symmetry2}, }
\int_{\R^2} V \, Z \, \phi \, dy=0 \Big\},
\end{equation*}
endowed with the inner product
\begin{equation*}
\begin{split}
[ \phi, \psi]
&=\int_{\R^2}\int_{\R^2}\frac{(\phi(x)-\phi(y))(\psi(x)-\psi(y))}{|x-y|^{2+2s}}dxdy.
\end{split}
\end{equation*}
Problem (\ref{linearproblem0}) expressed in weak form is equivalent to finding $\phi\in H$ such that
\begin{equation*}
[\phi,\psi]=\langle V \phi+h, \psi\rangle, \forall \ \psi \in H.
\end{equation*}

Using the Riesz's representation theorem, this equation gets rewritten in $H $ in the operational form
\begin{equation*}
\phi=(-\Delta)^{-s}( V\phi)+\tilde{h}:=K(\phi)+\tilde{h}
\end{equation*}
for a certain $\tilde{h}$ depends linearly in $h$ and $K$ is a compact operator in $H$. Fredholm's alternative guarantees unique solvability of this problem for any $h$ provided that the homogeneous equation
\begin{equation*}
\phi=K(\phi)
\end{equation*}
has only the zero solution in $H$. This follows from the apriori estimate in Proposition \ref{aprioriestimate0}.
\end{proof}

\section{Proof of Propositions~\ref{existenceofnonlinearproblem} and \ref{reduedproblem}}
\label{sect-nl}

\begin{proof}[Proof of Proposition~\ref{existenceofnonlinearproblem}]

Let $T_d$ be the operator obtained in Proposition \ref{existenceoflinearproblem0} that to $h$ with $h$ with $\|h\|_{**}<\infty$ and satisfying the symmetries \eqref{symmetry2} associates $\phi = T_d[h]$ solution to \eqref{linearproblem0}
with the estimate
\begin{align*}
\| T_d[h] \|_* \leq C_1  \|h\|_{**}.
\end{align*}

Let $X = \{ \, \phi \in L^\infty(\R^2) \, | \, \|\phi\|_*<\infty, \ \phi\text{ satisfies \eqref{symmetry2}} \,\} $ be endowed with $\| \ \| _*$.
Let
\begin{align*}
\mathcal A :X \to X , \quad
\mathcal{A}[\varphi]
= T_d[ -E+N[\phi] ] .
\end{align*}
Then \eqref{nonlinear} is equivalent to solve the fixed point problem
\[
\phi = \mathcal{A}[\phi] ,
\]
which we set-up in the closed ball
\begin{equation*}
\mathcal{B}=\{ \, \phi\in X \, | \, \|\phi\|_* \leq r_0 \, \},
\end{equation*}	
where $r_0>0$ is to be determined later.

From  \eqref{error} we see that
\begin{align*}
\|E\|_{**} \leq C \varepsilon^{3-2s}.
\end{align*}

We claim that if for $\phi \in \mathcal B$
\begin{align}
\label{estN1}
|N[\phi]| \leq C |\phi|^{\min(\gamma,2)}
\chi_{B_{R_0}(b_1')\cup B_{R_0}(b_2')} ,
\end{align}
where $R_0$ is a large fixed constant.
Indeed,  $N[\phi]$, defined in \eqref{def-N}, can be written as
\[
N[\phi] = N_1[\phi] + N_2[\phi]
\]
where
\begin{align*}
N_1[\phi]=
\chi_{B_{\delta/(\varepsilon\mu)}(b_1')}
\Big[
& \big( v_0 + \phi
+ c \varepsilon^{3-2s} \mu^{\frac{2s}{\gamma-1}+1} y_1
- \mu^{\frac{2s}{\gamma-1}} \lambda \big)_+^\gamma
\\
& \qquad
- \big( v_0
+ c \varepsilon^{3-2s} \mu^{\frac{2s}{\gamma-1}+1} y_1
- \mu^{\frac{2s}{\gamma-1}} \lambda \big)_+^\gamma
\\
& \qquad
- \gamma\big( v_0
+ c \varepsilon^{3-2s} \mu^{\frac{2s}{\gamma-1}+1} y_1
- \mu^{\frac{2s}{\gamma-1}} \lambda \big)_+^{\gamma-1} \phi
\Big] ,
\end{align*}
with an analogous formula for $N_2[\phi]$.
Using the definition of $v_0$ \eqref{def-v0} and the choice of $\lambda$ \eqref{lambda2} we see that
\begin{align*}
N_1[\phi]&=
\chi_{B_{\delta/(\varepsilon\mu)}(b_1')}
\Big[
( W(y-b_1') -1 + \mathcal R + \phi )_+^\gamma
-( W(y-b_1') -1 + \mathcal R  )_+^\gamma
\\
& \qquad
-\gamma( W(y-b_1') -1 + \mathcal R  )_+^{\gamma-1}\phi\Big],
\end{align*}
where $\mathcal R = O(\varepsilon^{3-2s}|y-b_1'|)$.
We deduce from here that
\[
|N_1[\phi]| \leq C |\phi|^{\min(\gamma,2)}
\]
and also that the support of $N_1[\phi] $ is contained in the ball $B_{R_0}(b_1')$ for some $R_0$ large fixed (assuming $\varepsilon>0$ is small). We have similar estimates for $N_2[\phi]$ and we deduce \eqref{estN1}.
From \eqref{estN1} we get
\[
\| N[\phi] \|_{**} \leq C_2 \|\phi\|_*^{\min(\gamma,2)}.
\]

So for $\phi \in \mathcal B$ we have
\begin{align*}
\|\mathcal A [\phi]\|_*
\leq C_1( \|E\|_{**} + \|N[\phi]\|_{**} )
\leq C_1 C \varepsilon^{3-2s} + C_1  C_2 r_0^{\min(\gamma,2)} .
\end{align*}
We choose $r_0>0$ small so that $C_1  C_2 r_0^{\min(\gamma,2)} \leq \frac{1}{2}r_0$. Then we work with $\varepsilon>0$ small so that $ C_1 C \varepsilon^{3-2s} \leq \frac{1}{2}r_0$.
This shows that $\mathcal A $ maps $\mathcal B$ into itself.

Also from the expression of $N$, we have
\begin{align*}
|N(\phi_1)-N(\phi_2)|
\leq C\big( |\phi_1|^{\min(\gamma-1,1)}
+|\phi_2|^{\min(\gamma-1,1)}
 \big) |\phi_1-\phi_2|	\chi_{B_{R_0}(b_1')\cup B_{R_0}(b_2')}
\end{align*}
where $R_0$ is a large fixed constant.
This implies that  for $\phi_1,\, \phi_2\in \mathcal{B}$,
\begin{equation*}
\|N(\phi_1)-N(\phi_2)\|_{**}\leq C
r_0^{\min(\gamma-1,1)}  \|\phi_1-\phi_2\|_*.
\end{equation*}	
Hence
\begin{equation*}
\|\mathcal{A}(\phi_1)-\mathcal{A}(\phi_2)\|_*\leq C_1 (\|N(\phi_1)-N(\phi_2)\|_{**})\leq
C C_1 r_0^{\min(\gamma-1,1)}  \|\phi_1-\phi_2\|_* .
\end{equation*}	
If $r_0$ is small we obtain that $\mathcal{A}$ is a contraction mapping from $\mathcal B$ into itslef and then problem (\ref{nonlinear}) admits a unique solution $\phi_d \in \mathcal B$.

From the proof above and the estimate for $E$, one has
\begin{equation*}
\|\phi_d\|_*\leq C\|E\|_{**}\leq C \ve^{3-2s} .
\end{equation*}	
Since $E$ and $V$ in \eqref{nonlinear} depend continuously on $d$, the fixed point characterization of $\phi_d$ shows that it is continuous with respect to $d$.

\end{proof}

\begin{proof}[Proof of Proposition~\ref{reduedproblem}]
We let $\phi$ denote the solution of \eqref{nonlinear} obtained in proposition~\ref{existenceofnonlinearproblem} and $v  = v_0 +\phi$.
Since
\[
(-\Delta)^s v - f(y,v) = L[\phi] + E - N[\phi]
\]
we have
\begin{align*}
\int_{\R^2}[(-\Delta)^s v_d -f(y,v_d)]\, Z\, dy
= \int_{\R^2} L[\phi]Z \, dy+ \int_{\R^2} E Z \, dy
+\int_{\R^2} N[\phi]Z\, dy.
\end{align*}

Let us consider the term $\int_{\R^2} E Z \, dy $.
From \eqref{error0} and the choice of $\lambda$ \eqref{lambda2} we have
\begin{align*}
E = E_1 +  E_2
\end{align*}
where
\begin{align*}
E_1(y)
&= \chi_{B_{R_0}(b_1')} \Big[ (W(y-b_1')-1)_+^\gamma
\\
& \qquad
- \Big( W(y-b_1')-1
+ W(b_1'-b_2')- W(y-b_2')
+ c \varepsilon^{3-2s}\mu^{\frac{2s}{\gamma-1}+1} (y_1-\frac{d}{\varepsilon\mu})
\Big)_+^\gamma \Big]
\end{align*}
and
\begin{align*}
E_2(y_1,y_2)
&= - E_1(y_1,y_2).
\end{align*}
Directly we have
\begin{align}
\label{Ej}
|E_j| \leq C \varepsilon^{3-2s}  \chi_{B_{R_0}(b_j')}.
\end{align}
Writing
\begin{align*}
\int_{\R^2} E Z \, dy
&=
\int_{\R^2} E_1 Z_1 \, dy
+\int_{\R^2} E_1 Z_2 \, dy
+\int_{\R^2} E_2 Z_1 \, dy
+\int_{\R^2} E_2 Z_2 \, dy
\end{align*}
we see that
\begin{align*}
\int_{\R^2} E_1 Z_1 \, dy
=\int_{\R^2}E_2 Z_2 \, dy,
\end{align*}
and
\begin{align*}
\int_{\R^2} E_1 Z_2 \, dy =
O( \varepsilon^{2(3-2s)}), \quad
\int_{\R^2} E_2 Z_1 \, dy =
O( \varepsilon^{2(3-2s)}),
\end{align*}
where we have used \eqref{Ej} and \eqref{asymptotic1}.
Therefore we need only to compute $ \int_{\R^2} E_1 Z_1 \, dy $.
For $y \in B_{R_0}(b_1')$ we have
\begin{align*}
E_1(y) &= -\gamma (W(y-b_1')-1)_+^{\gamma-1}
\Big[
W(b_1'-b_2')- W(y-b_2')
+ c \varepsilon^{3-2s}\mu^{\frac{2s}{\gamma-1}+1} (y_1-\frac{d}{\varepsilon\mu})
\Big]
\\
&\quad
+O(\varepsilon^{(3-2s)(\gamma-1)}) .
\end{align*}
So, integrating by parts
\begin{align*}
\int_{\R^2} E_1 Z_1 \, dy
&=
\int_{\R^2}
(W(y-b_1')-1)_+^\gamma
\left[
-\partial_{y_1} W(y-b_2')
+ c \varepsilon^{3-2s}
\mu^{\frac{2s}{\gamma-1}+1}
\right]\,dy
\end{align*}
and using the expansion \eqref{asymptotic}
\begin{align*}
\int_{\R^2}& E_1 Z_1  \, dy
\\
& =
\int_{\R^2}
(W(y-b_1')-1)_+^\gamma
\Big[
(2-2s) M_\gamma c_{2,s} \Big(\frac{\varepsilon\mu}{2d}\Big)^{3-2s}
+ c\varepsilon^{3-2s}
\mu^{\frac{2s}{\gamma-1}+1} +O(\varepsilon^{4-2s})
\Big]\,dy
\\
& = c_0
\varepsilon^{3-2s}
\Big[ \frac{1}{d^{3-2s}} + \frac{c}{m} 4 \pi \frac{\Gamma(s)}{\Gamma(2-s)}
+ O(\varepsilon)
\Big] ,
\end{align*}
for some constant $c_0\not=0$, where we have used
\eqref{masses}.
Therefore
\begin{align}
\label{projectionE}
\int_{\R^2} E Z \, d y =  2 c_0
\varepsilon^{3-2s}
\Big[ \frac{1}{d^{3-2s}} + \frac{c}{m} 4 \pi \frac{\Gamma(s)}{\Gamma(2-s)}
+ O(\varepsilon)
\Big] .
\end{align}

\medskip

Nest we consider $\int_{\R^2} N[\phi]Z\, dy$.
Using that $\|\phi\|_*\leq C \varepsilon^{3-2s}$
and \eqref{estN1}
we get
\begin{align}
\nonumber
\left|
\int_{\R^2} N[\phi]Z\, dy
\right|
&=
\left|
\int_{\R^2}
\big(  f(y,v_0+\phi) - f(y,v_0) - f_v(y,v_0)\phi
\big) Z\, dy
\right|
\\
\nonumber
& \leq
\int_{B_{R_0}(b_1')\cup B_{R_0}(b_2')}
|\phi|^{\min(\gamma,2)} Z\, dy
\\
\label{projectionofn}
& \leq C \varepsilon^{(3-2s)\min(\gamma,2) }.
\end{align}

Next, for the integral involving $L[\phi]$, we have
\begin{align*}
\int_{\R^2} L[\phi]Z\,dy
&= \int_{\R^2} \phi L[Z]\,dy
\\
&=
\int_{\R^2}\phi[\gamma(W(y-b_1')-1)^{\gamma-1}_+Z_1+\gamma (W(y-b_2')-1)_+^{\gamma-1}Z_2-f_v(y,v_0)Z]dy .
\end{align*}
But
\begin{align*}
& \gamma(W(y-b_1')-1)^{\gamma-1}_+Z_1+\gamma (W(y-b_2')-1)_+^{\gamma-1}Z_2-f_v(y,v_0)Z
= A_1 + A_2
\end{align*}
where
\begin{align*}
A_1 &=
\gamma(W(y-b_1')-1)^{\gamma-1}_+Z_1
- \gamma ( W(y-b_1')-1 + \mathcal R_1)_+^\gamma \chi_{B_{R_0}(b_1')}( Z_1+Z_2)
\\
A_1 &=
\gamma(W(y-b_2')-1)^{\gamma-1}_+Z_2
- \gamma ( W(y-b_2')-1 + \mathcal R_1)_+^\gamma \chi_{B_{R_0}(b_2')}( Z_1+Z_2)
\end{align*}
and $\mathcal R_1 = O(\varepsilon^{3-2s} |y-b_1'|)$ and $\mathcal R_2 = O(\varepsilon^{3-2s} |y-b_2'|)$.
Using the decay of $W'$ \eqref{asymptotic1} we find that
\begin{align*}
|A_1|\leq C \varepsilon^{(3-2s)\min(\gamma-1,1)}
\chi_{B_{R_0}(b_1')} , \quad
|A_2|\leq C  \varepsilon^{(3-2s)\min(\gamma-1,1)}
\chi_{B_{R_0}(b_2')} ,
\end{align*}
for a possible larger $R_0$.
Using that $\|\phi\|_*\leq C \varepsilon^{3-2s}$ we find that
\begin{align}
\label{projectionLphi}
\left|
\int_{\R^2} L[\phi]Z\,dy
\right| \leq C\varepsilon^{(3-2s)\min(\gamma,2) }.
\end{align}

Putting together \eqref{projectionE}, \eqref{projectionofn} and \eqref{projectionLphi} we obtain the desired conclusion.

\end{proof}

\section{On Theorems~\ref{thm-rot}--\ref{thm-rot2}}
\label{sect-extensions}

The proofs of the remaining results follow similar lines as those  above, so we only present a sketch of the necessary changes.

\medskip
Concerning Theorem \ref{thm-multiple-trav}, let us first  formally  derive the balancing conditions (\ref{bal}). As described in section~\ref{sect-scheme}, we
consider the elliptic problem
\begin{align}
\label{eq-trav-multiple}
\left\{
\begin{aligned}
(-\Delta)^s \psi &=
\varepsilon^{(2-2s)\gamma-2}
\sum_{j=1}^k ( \psi + c x_1  - \varepsilon^{2s-2} \lambda_j^+)_+^\gamma \chi_{B_\delta(p_j)}
\\
& \quad -
\varepsilon^{(2-2s)\gamma-2}
\sum_{l=1}^k ( -\psi - c x_1  - \varepsilon^{2s-2} \lambda_l^-)_+^\gamma \chi_{B_\delta(q_l)}
\quad \text{in }\R^2,
\\
\psi(x) & \to 0 \quad \text{as }|x|\to \infty,
\end{aligned}
\right.
\end{align}
and  look a solution that at main order is approximated by
\begin{align*}
\psi_0(x) =
\varepsilon^{2s-2}\mu^{-\frac{2s}{\gamma-1}}
\Big[
\sum_{j=1}^k
W\Big( \frac{x-p_j}{\varepsilon\mu}\Big)
-
\sum_{l=1}^k
W\Big( \frac{x-q_l}{\varepsilon\mu}\Big)\Big],
\end{align*}
where, as in \eqref{mj}, $\mu>0$ is such that $M_\gamma \mu^{2(1-\frac{s\gamma}{\gamma-1})}=1$, and $\lambda_j^+$, $\lambda_l^-$ are as in \eqref{eq-lambda}, \eqref{eq-lambda2}.
With these choices, the error of approximation, defined by
\begin{align*}
E & = (-\Delta)^s \psi
-\varepsilon^{(2-2s)\gamma-2}
\sum_{j=1}^k ( \psi + c x_1  - \varepsilon^{2s-2} \lambda_j^+)_+^\gamma \chi_{B_\delta(p_j)}
\\
& \quad +\varepsilon^{(2-2s)\gamma-2}
\sum_{l=1}^k ( -\psi - c x_1  - \varepsilon^{2s-2} \lambda_l^-)_+^\gamma \chi_{B_\delta(q_l)}
\end{align*}
has the form, for $x$ near $p_i$:
\begin{align*}
E &=
\varepsilon^{-2} \mu^{-\frac{2s\gamma}{\gamma-1}}
\Big[
\Big( W\Big(\frac{x-p_i}{\varepsilon \mu}\Big) -1 \Big)_+^\gamma
\\
\nonumber
& \qquad
-
\Big(
W\Big(\frac{x-p_i}{\varepsilon \mu}\Big)
+  \sum_{j\not=i}
\Big( W\Big(\frac{x-p_j}{\varepsilon \mu}\Big)
-W\Big(\frac{p_i-p_j}{\varepsilon \mu}\Big)\Big)
\\
& \qquad
- \sum_{l=1}^k
\Big( W\Big(\frac{x-q_l}{\varepsilon \mu}\Big)
-W\Big(\frac{p_i-q_l}{\varepsilon \mu}\Big)\Big)
\\
& \qquad
+ c \mu^{\frac{2s}{\gamma-1}} \varepsilon^{2-2s} (x_1 -p_{i,1})
\Big)_+^\gamma
\Big] .
\end{align*}
Changing $x = \varepsilon \mu y$ and expanding in  $\varepsilon$ gives
\begin{align*}
&   \sum_{j\not=i}
\Big( W\Big(\frac{x-p_j}{\varepsilon \mu}\Big)
-W\Big(\frac{p_i-p_j}{\varepsilon \mu}\Big)\Big)
- \sum_{l=1}^k
\Big( W\Big(\frac{x-q_l}{\varepsilon \mu}\Big)
-W\Big(\frac{p_i-q_l}{\varepsilon \mu}\Big)\Big)
\\
& \qquad
+ c \mu^{\frac{2s}{\gamma-1}} \varepsilon^{2-2s} (x_1 -p_{i,1})
\\
& \sim
c
\varepsilon^{3-2s}
\left\{
- \sum_{j \not = i} \frac{ (p_i-p_j)\cdot y}{ |p_i-p_j|^{4-2s}} + \sum_{l=1}^k \frac{(p_i-q_l)\cdot y}{|p_i-q_l|^{4-2s}} + \frac{c}{m} \frac{2^{2s-1}\pi \Gamma(s)}{\Gamma(2-s)}  y\cdot e_1
 \right\} + O(\varepsilon^{4-2s})
 .
\end{align*}
We want that the first order expansion vanishes, which leads to the equation
\begin{align*}
\sum_{j \not = i} \frac{ p_i-p_j}{ |p_i-p_j|^{4-2s}} - \sum_{l=1}^k \frac{p_i-q_l}{|p_i-q_l|^{4-2s}} =  c \frac{2^{2s-1}\pi \Gamma(s)}{\Gamma(2-s)} e_1,
\end{align*}
for any $i=1,\ldots,k$.

A similar computation for $x$ near $q_m$
leads to
\begin{align*}
\sum_{ l\not =m  } \frac{q_m-q_l}{ |q_m-q_l|^{4-2s}} -\sum_{j=1}^k \frac{q_m-p_j}{ |q_m-p_j|^{4-2s}} & =- c \frac{2^{2s-1}\pi \Gamma(s)}{\Gamma(2-s)} e_1 .
\end{align*}

\hide{
\color{red}

At the point $p_i$, we let $ x=p_i +y$. Then the first error can be computed as follows
$$  E=- (-\Delta)^s [\sum_{j=1}^k \phi (x- p_j) - \sum_{l=1}^k \phi (x-q_l)]  + f(\sum_{j=1}^k \phi (x- p_j) - \sum_{l=1}^k \phi (x-q_l) - \ve x_1) $$
$$= f(\sum_{j=1}^k \phi (x- p_j) - \sum_{l=1}^k \phi (x-q_l) - \ve x_1) - \sum_{j=1}^k f(\phi (x-p_j)) +\sum_{l=1}^k f( \phi (x-q_l) ) $$

At the point $p_i$, we let $ x=p_i +y$. Then for $|y| $ bounded, by the asymptotics in (\ref{asymptotic}), we have
$$ \sum_{j=1}^k \phi (x- p_j) - \sum_{l=1}^k \phi (x-q_l) -\ve x_1
$$
$$ \sim  \phi (y) + \sum_{j \not = i} \frac{M c_{2,s}}{ |p_i-p_j+y|^{2-2s}}-\sum_{l=1}^k \frac{M c_{2,s}}{ | p_i-q_l +y|^{2-2s}} - \ve (p_i+ y)$$
The translating mode in the above error must be balanced. This gives
$$ \sum_{j \not = i} \frac{ p_i-p_j}{ |p_i-p_j|^{4-2s}} - \sum_{l} \frac{p_i-q_l}{|p_i-q_l|^{4-2s}}= \mu$$
where $ \mu =-\frac{1}{ 2(1-s) M c_{2,s}}$. Similarly at a negative vortex point $q_m$, we obtain
$$ \sum_{j } \frac{ q_m-p_j}{ |q_m-p_j|^{4-2s}} - \sum_{l \not = m} \frac{q_l-q_m}{|q_l-q_m|^{4-2s}}= -\mu$$
This gives necessary conditions for the existence of multiple vortices.
}

To prove that if  $({\bf p}, {\bf q})$ is a nondegenerate symmetric array of traveling vortices there exists a solution of \eqref{eq-trav-multiple} close to $\psi_0$, we work in the symmetry class
\begin{equation}\label{symmetry1}
\Psi(x_1,x_2)=-\Psi(-x_1,x_2)=\Psi(x_1,-x_2).
\end{equation}
Following the same proof as in that of Theorem \ref{thm-vortex-pair} and  utilizing  the non-degeneracy conditions under the symmetry condition (\ref{symmetry1}), which is guaranteed by the symmetry condition (\ref{sym1}) on $({\bf p}, {\bf q})$, we obtain Theorem \ref{thm-multiple-trav}.


\medskip
The proof of Theorem~\ref{thm-trav-general} is similar.
In that case final the adjustment of the points
${\bf b} = (b_1,\ldots, b_j)$
as a small perturbation of a given ${\bf b}_0$, critical point of $ I (\bf b)$  non-degenerate up to vertical translations,
obeys
an equation of the form
$$
\nabla_{\bf b}  I (\bf b) +  \mathcal N (\bf b) = 0
$$
where $ \mathcal N (\bf b)$ is an $\ve$- small term, which is invariant under vertical translations. A standard degree argument involving a local orthogonal decomposition of $\bf b$ yields the desired result.

\medskip
\medskip
In the case of the rotating solutions as  in
 Theorems~\ref{thm-rot} and \ref{thm-rot2}, we need to find a family of solutions  $ \Theta_\ve(x)$ to the equation
\begin{align*}
(\nabla^\perp \Psi + \alpha x^\perp )\cdot \nabla \Theta = 0 , \quad \Psi = (-\Delta)^{-s} \Theta,
\end{align*}
such that
$$
 \Theta_\ve(x) \rightharpoonup \sum_{j=1}^k m_j\delta (x-b_j^0) ,
$$
for given intensities $m_j$ and a solution ${\bf b}^0= (b_1^0,\ldots, b_k^0)$ of system \eqref{reduced-rot}.
To achieve this we consider the  elliptic problem
\begin{align*}
\label{eqrot}
\left\{
\begin{aligned}
(-\Delta)^s \psi &=
\varepsilon^{(2-2s)\gamma-2}
\sum_{j=1}^k \sigma_j \Big( \sigma_j \Big( \psi + \alpha \frac{|x|^2}{2} \Big) - \varepsilon^{2s-2} \lambda_j\Big)_+^\gamma \chi_{B_\delta(b_j)}
\quad \text{in }\R^2,
\\
\psi(x) & \to 0 \quad \text{as }|x|\to \infty,
\end{aligned}
\right.
\end{align*}
where $1<\gamma<\frac{2+2s}{2-2s}$,
$\gamma \not=\frac{1}{1-s}$,
$\sigma_j= +1$  if  $m_j>0$ and $=-1$ if $m_j<0$. The choice of $\lambda_j$ is done similarly as in the case of the traveling solutions and we have $\lambda_j =  \mu_j^{-\frac{2s}{\gamma-1}} + O(\varepsilon^{2-2s})$.
The points  $b_1,\ldots,b_k$ are close to $b_1^0,\ldots,b_k^0$, and  $\delta>0$ is fixed so that the balls $B_\delta(b_j)$ are disjoint.

The ansatz $\psi_0$ is the same as in \eqref{general-ansatz} with $\mu$ as in \eqref{mj}.
The proof of Theorem~\ref{thm-rot} is then a direct adaptation of the proof of Theorem~\ref{thm-vortex-pair}.

 Theorem \ref{thm-rot2} similarly follows after a reduction to a problem of the form
$$
\nabla_{\bf b}  J (\bf b) +  \mathcal N (\bf b) = 0
$$
where now $\mathcal N (\bf b)$ is a small $\ve$-perturbation which is invariant under rotations.

\medskip\noindent{\bf Acknowledgments:}
W. Ao is partially supported by NSF of China. J.~D\'avila has been supported  by  a Royal Society  Wolfson Fellowship, UK and  Fondecyt grant 1170224, Chile.
 M.~del Pino has been supported by a Royal Society Research Professorship, UK.
 M. Musso has been supported by EPSRC research Grant EP/T008458/1.
 The  research  of J.~Wei is partially supported by NSERC of Canada.


\begin{thebibliography}{10}

\bibitem{ao_ode-methods_2020} Weiwei Ao, Hardy Chan, Azahara DelaTorre, Marco A. Fontelos, Mara Del Mar Gonzlez, and
Juncheng Wei, {\em ODE-methods in non-local equations}, arXiv: 1910.14512, 2020.


\bibitem{bandle_harmonic_1996} C. Bandle and M. Flucher, {\em Harmonic radius and concentration of energy; hyperbolic radius and Liouville’s
equations $\Delta U = e^U$ and $\Delta U = U^{(n+2)/(n−2)}$, } SIAM Rev. 38 (1996), no. 2, 191–238.


\bibitem{cabre-sire} Xavier Cabre and Yannick Sire, ´ {\em Nonlinear equations for fractional Laplacians, I: Regularity, maximum
principles, and Hamiltonian estimates}, Ann. Inst. H. Poincare Anal. Non Lin ´ eaire ´ 31 (2014), no. 1,
23–53.


 \bibitem{caffarelli-silvestre-extension}  Luis Caffarelli and Luis Silvestre, {\em An extension problem related to the fractional Laplacian}, Comm. Partial Differential Equations 32 (2007), 1245–1260.


\bibitem{caffarelli_asymptotic_1980} Luis A. Caffarelli and Avner Friedman, {\em Asymptotic estimates for the plasma problem}, Duke Mathematical Journal 47 (1980), no. 3, 705–742.

\bibitem{campos-delpino-dolbeault}
 Juan Campos, Manuel del Pino, and Jean Dolbeault, {\em Relative equilibria in continuous stellar dynamics},
Comm. Math. Phys. 300 (2010), no. 3, 765–788.

\bibitem{CLW2014} Daomin Cao, Zhongyuan Liu, and Juncheng Wei, {\em Regularization of point vortices pairs for the Euler
equation in dimension two}, Arch. Ration. Mech. Anal. 212 (2014), no. 1, 179–217.

\bibitem{carrillo_ground_2018} Jose A. Carrillo, Franca Hoffmann, Edoardo Mainini, and Bruno Volzone, ´{\em  Ground states in the
diffusion-dominated regime}, Calc. Var. Partial Differential Equations 57 (2018), no. 5, Paper No. 127,
28.


\bibitem{castro_global_2017} Angel Castro, Diego Cordoba, and Javier G ´ omez-Serrano, ´{\em  Global Smooth Solutions for the Inviscid
SQG Equation}, Mem. Amer. Math. Soc. 266 (2020), no. 1292.


\bibitem{chae-et-al-2012} Dongho Chae, Peter Constantin, Diego Cordoba, Francisco Gancedo, and Jiahong Wu, ´ {\em Generalized
surface quasi-geostrophic equations with singular velocities}, Comm. Pure Appl. Math. 65 (2012), no. 8,
1037–1066.

\bibitem{chan_uniqueness_2020} Hardy Chan, Mara del Mar Gonzlez, Yanghong Huang, Edoardo Mainini, and Bruno Volzone,
{\em Uniqueness of entire ground states for the fractional plasma problem}, arXiv: 2003.01093, 2020.


 \bibitem{Constantin_1994} Peter Constantin, Andrew J. Majda, and Esteban Tabak, {\em Formation of strong fronts in the 2-D quasigeostrophic thermal active scalar}, Nonlinearity 7 (1994), no. 6, 1495–1533.

\bibitem{davila_supercritical_2007} Juan Davila, Manuel del Pino, and Monica Musso, ´{\em  The supercritical Lane-Emden-Fowler equation in
exterior domains}, Comm. Partial Differential Equations 32 (2007), no. 7-9, 1225–1243.

\bibitem{davila_fast_2008} Juan Davila, Manuel del Pino, Monica Musso, and Juncheng Wei, ´{\em  Fast and slow decay solutions for
supercritical elliptic problems in exterior domains}, Calc. Var. Partial Differential Equations 32 (2008),
no. 4, 453–480.


\bibitem{davila_gluing_2020-1} Juan Davila, Manuel Del Pino, Monica Musso, and Juncheng Wei, ´ {\em Gluing Methods for Vortex Dynamics in Euler Flows}, Arch. Ration. Mech. Anal. 235 (2020), no. 3, 1467–1530.

\bibitem{SS2013} Sebastien de Valeriola and Jean Van Schaftingen, ´ {\em Desingularization of vortex rings and shallow water
vortices by a semilinear elliptic problem}, Arch. Ration. Mech. Anal. 210 (2013), no. 2, 409–450.


\bibitem{fall} Mouhamed Moustapha Fall, Entire s-harmonic functions are affine, Proc. Amer. Math. Soc. 144 (2016),
no. 6, 2587–2592.

\bibitem{flucher_asymptotic_1998} M. Flucher and J. Wei, {\em Asymptotic shape and location of small cores in elliptic free-boundary problems},
Math. Z. 228 (1998), no. 4, 683–703.

\bibitem{frank_uniqueness_2016} Rupert L. Frank, Enno Lenzmann, and Luis Silvestre, {\em Uniqueness of radial solutions for the fractional
Laplacian}, Comm. Pure Appl. Math. 69 (2016), no. 9, 1671–1726.


 \bibitem{gravejat_smooth_2019} Philippe Gravejat and Didier Smets, {\em Smooth travelling-wave solutions to the inviscid surface quasigeostrophic equation}, Int. Math. Res. Not. IMRN (2019), no. 6, 1744–1757.

\bibitem{kiselev_finite_2016} Alexander Kiselev, Lenya Ryzhik, Yao Yao, and Andrej Zlato, {\em Finite time singularity for the modified
sqg patch equation}, Annals of Mathematics 184 (2016), no. 3, 909–948.

\bibitem{kiselev_local_2017} Alexander Kiselev, Yao Yao, and Andrej Zlato,  {\em Local regularity for the modified sqg patch equation},
Communications on Pure and Applied Mathematics 70 (2017), no. 7, 1253–1315.

\bibitem{liu_multi-vortex_2018} Yong Liu and Juncheng Wei, {\em Multi-vortex traveling waves for the gross-pitaevskii equation and the adlermoser polynomials}, SIAM J. Math. Anal. 52 (2020), no. 4, 3546–3579.


 \bibitem{marchioro_euler_1983} C. Marchioro and M. Pulvirenti, {\em Euler evolution for singular initial data and vortex theory}, Comm.
Math. Phys. 91 (1983), no. 4, 563–572.


 \bibitem{moulton} F. R. Moulton, {\em The straight line solutions of the problem of n bodies}, Ann. of Math. (2) 12 (1910), no. 1,
1–17.

\bibitem{palmore-class1} Julian I. Palmore, {\em Classifying relative equilibria. I,} Bull. Amer. Math. Soc. 79 (1973), 904–908.

\bibitem{palmore-class2}  Julian I. Palmore,{\em  Classifying relative equilibria. II,} Bull. Amer. Math. Soc. 81 (1975), 489–491.

\bibitem{palmore-class3}  Julian I. Palmore, {\em  Classifying relative equilibria. III,}  Lett. Math. Phys. 1 (1975/76), no. 1, 71–73.

\bibitem{palmore-annals} JUlian I. Palmore, {\em  Measure of degenerate relative equilibria. I,} Ann. of Math. (2) 104 (1976), no. 3, 421–429.

\bibitem{perko-walter} L. M. Perko and E. L. Walter, {\em Regular polygon solutions of the N-body problem}, Proc. Amer. Math. Soc.
94 (1985), no. 2, 301–309.

 \bibitem{rosenzweig} Matthew Rosenzweig, {\em Justification of the point vortex approximation for modified surface quasigeostrophic equations}, SIAM J. Math. Anal. 52 (2020), no. 2, 1690–1728.

\bibitem{smets_desingularization_2010} Didier Smets and Jean Van Schaftingen, {\em Desingularization of vortices for the Euler equation}, Arch.
Ration. Mech. Anal. 198 (2010), no. 3, 869–925.

\bibitem{temam_non-linear_1975} R. Temam, {\em A non-linear eigenvalue problem: the shape at equilibrium of a confined plasma}, Arch. Rational
Mech. Anal. 60 (1975/76), no. 1, 51–73.

\bibitem{temam_remarks_1977} R. Temam, {\em Remarks on a free boundary value problem arising in plasma physics}, Comm. Partial Differential
Equations 2 (1977), no. 6, 563–585.




\end{thebibliography}
\end{document}